\documentclass{amsart}

\usepackage{mathtools}
\usepackage{amsmath,amssymb,amsthm,amsfonts,mathrsfs}
\usepackage[dvipsnames]{xcolor}
\usepackage{graphpap,paralist}
\usepackage[mathscr]{eucal}
\usepackage{mathabx}
\usepackage[pdftex,colorlinks,citecolor=blue]{hyperref}
\usepackage{yfonts}
\usepackage{enumerate}
\usepackage[alphabetic]{amsrefs}
\usepackage{array}

\makeatletter
\newcommand{\thickhline}{%
    \noalign {\ifnum 0=`}\fi \hrule height 1.25pt
    \futurelet \reserved@a \@xhline
}
\newcolumntype{"}{@{\hskip\tabcolsep\vrule width 1.25pt\hskip\tabcolsep}}
\makeatother

\usepackage{caption}
\usepackage{subcaption}

\usepackage{thm-restate}

\usepackage[normalem]{ulem}

\usepackage[margin=1in]{geometry}

\numberwithin{equation}{section}

\theoremstyle{definition}
\newtheorem{theorem}{Theorem}[section]
\newtheorem{definition}[theorem]{Definition}
\newtheorem{conjecture}[theorem]{Conjecture}
\newtheorem{lemma}[theorem]{Lemma}
\newtheorem{proposition}[theorem]{Proposition}
\newtheorem{corollary}[theorem]{Corollary}
\newtheorem*{theorem*}{Theorem}
\theoremstyle{remark}
\newtheorem{remark}[theorem]{Remark}
\newtheorem{example}[theorem]{Example}

\newcommand{\supp}{\text{supp}}

\newcommand{\sfn}{\mathsf{n}}
\newcommand{\sfe}{\mathsf{e}}

\DeclareMathOperator{\Newton}{Newton}
\DeclareMathOperator{\gr}{gr}

\DeclareMathOperator{\conv}{conv}

\DeclareMathOperator{\perm}{perm}

\DeclareMathOperator{\wt}{wt}

\newcommand{\R}{\mathbb{R}}

\title[The Newton polytope and Lorentzian property of chromatic symmetric functions]{The Newton polytope and Lorentzian property \\ of chromatic symmetric functions}

\author{Jacob P. Matherne}
\address{Department of Mathematics, North Carolina State University, Raleigh, NC., USA}
\email{jpmather@ncsu.edu}

\author{Alejandro H. Morales}
\address{LACIM, D\'epartement de Math\'ematiques, Universit\'e du Qu\'ebec \`a Montr\'eal, Montr\'eal, QC, Canada}
\email{morales\_borrero.alejandro@uqam.ca}

\author{Jesse Selover}
\address{Department of Mathematics and Statistics, University of Massachusetts, Amherst, MA., USA}
\email{jselover@math.umass.edu}

\begin{document}

\begin{abstract}
Chromatic symmetric functions are well-studied symmetric functions in algebraic combinatorics that generalize the chromatic polynomial and are related to Hessenberg varieties and diagonal harmonics.  Motivated by the Stanley--Stembridge conjecture, we show that the allowable coloring weights for indifference graphs of Dyck paths are the lattice points of a permutahedron $\mathcal{P}_\lambda$, and we give a formula for the dominant weight $\lambda$.  Furthermore, we conjecture that such chromatic symmetric functions are Lorentzian, a property introduced by Br\"and\'en and Huh as a bridge between discrete convex analysis and concavity properties in combinatorics, and we prove this conjecture for abelian Dyck paths.  We extend our results on the Newton polytope to incomparability graphs of (3+1)-free posets, and we give a number of conjectures and results stemming from our work, including results on the complexity of computing the coefficients and relations with the $\zeta$ map from diagonal harmonics.
\end{abstract}

\maketitle

%%%%%%%%%%%%%%%%%%%%%%%%%%%%%%%%%%%%%%%%%%%%%%%%%%%%%%%%%%%%%%%%%%%%%%%%%%%
\section{Introduction} \label{sec:intro}
%%%%%%%%%%%%%%%%%%%%%%%%%%%%%%%%%%%%%%%%%%%%%%%%%%%%%%%%%%%%%%%%%%%%%%%%%%%

%%%%%%%%%%%%%%%%%%%%%%%%%%%%%%%%%%%%%%%%%%%%%%%%%%%%%%%%%%%%%%%%%%%%%%%%%%%

\subsection{Motivation}
%%%%%%%%%%%%%%%%%%%%%%%%%%%%%%%%%%%%%%%%%%%%%%%%%%%%%%%%%%%%%%%%%%%%%%%%%%%

The study of proper colorings of
a graph $G$ is a fundamental topic in graph
theory and theoretical computer science. For a fixed graph $G$, the number of proper colorings of $G$ with $n$
colors is given by the chromatic polynomial $\chi_G(n)$, which was first introduced by Birkhoff in 1912 for planar graphs \cite{b12} and then for all graphs in 1932 by Whitney \cite{w32}. In 1995, Stanley defined the following symmetric function
generalization of $\chi_G(n)$: for a graph $G=(V,E)$ let 
\[
X_G({\bf x}) \coloneq \sum_{\overset{f\colon V(G) \to \mathbb{N}}{f \text{ proper}}} x_{f(1)} x_{f(2)}\cdots,
\] 
where the sum is over all proper
colorings $f$ of $G$ (i.e., colorings satisfying $f(i)\neq f(j)$ if $(i,j)$ is an
edge of $G$) \cite{StChrom1}. This symmetric function has connections to
 combinatorial Hopf algebras \cite{ABS}, topology \cite{SY,CSSY}, statistical mechanics \cite{McMo}, representation theory \cite{HaPre,CM,FennSommers}, and  algebraic
geometry \cite{BrCh}.

There are fundamental open questions about $X_G({\bf x})$. For
instance, Stanley conjectured in \cite{StChrom1,StChrom2} that $X_G({\bf x})$
expands positively in the Schur function basis (i.e. is $s$-positive) whenever $G$ is claw-free.

\begin{conjecture}[Stanley \cite{StChrom1}]\label{conj:stan}
If $G$ is claw-free, then $X_G({\bf x})$ is $s$-positive.
\end{conjecture}

Stanley verified this conjecture for {\em co-bipartite graphs}, which are complements of bipartite graphs. Another conjecture of Stanley and Stembridge that motivates this paper states that if the graph $G=G(P)$ is the {\em incomparability graph} of a (3+1)-free
poset $P$ (a poset with no subposet consisting of a 3-chain and an
incomparable element), then $X_{G(P)}({\bf x})$ expands positively in the
elementary basis of symmetric functions (i.e. is $e$-positive) \cite{StSt}. 
Gasharov proved that $X_{G(P)}({\bf x})$ is $s$-positive \cite{G2}, which is implied by both conjectures since such graphs $G(P)$ are claw-free and $e$-positivity implies $s$-positivity. Lewis--Zhang in \cite{LZ} and Guay-Paquet--Morales--Rowland  in \cite{GPMR} studied the enumeration and structure of (3+1)-free posets. In \cite{MGP}, Guay-Paquet used this work and the {\em modular relation} of $X_{G(P)}({\bf x})$ 
to reduce this conjecture to a subfamily of graphs
in bijection with Dyck paths \cite{MGP}, one of the hundreds of objects counted
by the Catalan numbers. Given a Dyck path $d$ from $(0,0)$ to $(n,n)$, the \emph{indifference graph}  $G(d)$ of the Dyck path $d$ has vertices $\{1,\ldots, n\}$ and edges  $(i,j)$ with $i<j$, if the cell
$(i,j)$ is below the path $d$ (see Figure~\ref{fig:ex_path_graph}).

\begin{conjecture}[Stanley--Stembridge \cite{StSt}] \label{conj: StaSteShWa}
For any Dyck path $d$,  $X_{G(d)}({\bf x})$ is $e$-positive.\footnote{A proof of this conjecture was announced by Hikita in \cite{hikita2024proofstanleystembridgeconjecture}.}
\end{conjecture}

Recently, $X_{G(d)}({\bf x})$ has been related to other very interesting
mathematical objects: the representation theory of {\em Hessenberg varieties} (see e.g., \cite{ShW2}) and the space of diagonal harmonics \cite{CM,HW,AP,AS}. For other classes of graphs with $e$-positive chromatic symmetric function, see \cite{FHM}.

%%%%%%%%%%%%%%%%%%%%%%%%%%%%%%%%%%%%%%%%%%%%%%%%%%%%%%%%%%%%%%%%%%%%%%%%%%%
\subsection{Main results}
%%%%%%%%%%%%%%%%%%%%%%%%%%%%%%%%%%%%%%%%%%%%%%%%%%%%%%%%%%%%%%%%%%%%%%%%%%%

The purpose of this paper is to contribute to the study of three classes of chromatic symmetric functions: those of co-bipartite graphs $G$, of indifference graphs $G(d)$ of Dyck paths $d$, and of incomparability graphs $G(P)$ of (3+1)-free posets $P$.  We note that indifference graphs of Dyck paths are the same as incomparability graphs of posets that are both (3+1)-free and (2+2)-free (such posets are sometimes called \emph{unit interval orders}, e.g., see \cite[Section 2.1]{AP}), and co-bipartite graphs are incomparability graphs of 3-free posets.  Thus, the class of incomparabilty graphs of (3+1)-free posets contains both the class of co-bipartite graphs and the class of indifference graphs of Dyck paths (see \cite{GPMR,MGP}).

We study these three classes via their Newton polytopes.  Recall that the Newton polytope of a multivariate polynomial $p({\bf x}) \in \mathbb{R}[x_1,\ldots,x_k]$ is the convex hull in $\mathbb{R}^k$ of the support of $p$, and that $p$ is said to be SNP if its Newton polytope is saturated, i.e. if the support of $p$ is equal to the set of lattice points in the Newton polytope of $p$ \cite{MTY}.  Our main result is that the chromatic symmetric functions in each of the three classes that we study are SNP, and moreover, their Newton polytopes are explicitly described permutahedra. In what follows, for a partition $\lambda$ of length $\ell$ and a nonnegative integer $k \geq \ell$, the permutahedron $\mathcal{P}^{(k)}_{\lambda}$ is the convex hull of permutations of $(\lambda_1,\ldots,\lambda_{\ell},0,\ldots,0)$ in $\mathbb{R}^k$. If $k<\ell$, the permutahedron $\mathcal{P}^{(k)}_{\lambda}$ is the empty set.

\begin{theorem}\label{thm:main}\
\begin{enumerate}
    \item (Proposition~\ref{co-bipartite-m-convex}) For $G$ a co-bipartite graph, $X_G(x_1,\ldots,x_k)$ is SNP and its Newton polytope is the permutahedron $\mathcal{P}^{(k)}_{\lambda(G)}$.
    \item (Theorem~\ref{thm:p_lambda}) For $d$ a Dyck path, $X_{G(d)}(x_1,\ldots,x_k)$ is SNP and its Newton polytope is the permutahedron  $\mathcal{P}^{(k)}_{\lambda(d)}$.
    \item (Theorem~\ref{thm:p_lambda 3+1 free}) For $G(P)$ the incomparability graph of a (3+1)-free poset $P$, $X_{G(P)}(x_1, \ldots, x_k)$ is SNP and its Newton polytope is the permutahedron $\mathcal{P}^{(k)}_{\lambda(P)}$.
\end{enumerate}
\end{theorem}

In each of the three cases of Theorem~\ref{thm:main}, we explicitly describe the Newton polytope: for example, in case (2), $\lambda(d)$ is the weight $\lambda^{\gr}(d)$ of the greedy coloring of $G(d)$. Interestingly, the weight $\lambda^{\gr}(d)$ appears in representation theory, where it is the partition arising from the Jordan form of the unique nilpotent orbit associated to a given ad-nilpotent ideal $I$ of the set of strictly upper triangular matrices \cite{Gerstenhaber} (see Remarks~\ref{rem:gers} and~\ref{rem:coloringlemma}).  For more on this connection, including the relation between Dyck paths and ad-nilpotent ideals, we point to \cite[Section 6]{FennSommers}.

The proof of Theorem~\ref{thm:main} in each case proceeds by finding a special coloring \(\text{gr}\). These symmetric functions are by definition positive in the monomial basis, and all three classes of graphs have Stanley's {\em nice} property (see Section~\ref{sec:nice}), so the support will contain any integer vector dominated by the weight of \(\text{gr}\). To complete the proof, we prove that any vector in the support is dominated by the weight of \(\text{gr}\).

We reiterate that case (3) implies both cases (1) and (2). While preparing the manuscript, the authors learned that one of the main ingredients for case (2), Lemma~\ref{lemma greedy is max}, 
was already known to Chow in an unpublished
note \cite{TC} (see Remark~\ref{rem:coloringlemma}). Cases (1) and (3) are new, and case (3) requires an in-depth study of the structure of (3+1)-free posets from \cite{GPMR,MGP} and the recent modular relation of Guay-Paquet \cite{MGP} and Orellana--Scott \cite{OS}.

%%%%%%%%%%%%%%%%%%%%%%%%%%%%%%%%%%%%%%%%%%%%%%%%%%%%%%%%%%%%%%%%%%%%%%%%%%%
\subsection{Applications and conjectures}
%%%%%%%%%%%%%%%%%%%%%%%%%%%%%%%%%%%%%%%%%%%%%%%%%%%%%%%%%%%%%%%%%%%%%%%%%%%

Here we highlight some of the consequences of Theorem~\ref{thm:main} that appear in detail in Sections~\ref{sec:stabandlor} and ~\ref{sec: conjecture}.

%%%%%%%%%%
\subsubsection{M-convexity and the Lorentzian property}

One strengthening of the SNP property is  M-convexity, a property that first appeared in discrete convex analysis \cite{M03}.  Write $e_i$ for the $i$th standard unit vector in $\mathbb{N}^k$. A subset $J$ of $\mathbb{N}^k$ is {\em
  matroid-convex} or {\em M-convex} if for all $\alpha,\beta \in J$
and for all $i$ such that $\alpha_i >\beta_i$, there exists a $j$ such
that $\beta_j >\alpha_j$ and $\alpha - e_i+e_j \in J$.  The convex hull of any M-convex
set is a {\em generalized
permutahedron} \cite{PosGP}, and the set of lattice points of an integral generalized permutahedron is an M-convex set \cite[Theorem 1.9]{M03}. Let $H_k^n$ be the set of degree $n$ homogeneous polynomials in $\R[x_1,\ldots,x_k]$.  A
polynomial $f$ in $H_k^n$ is M-convex if its support is
M-convex. Note that if $f$ is M-convex, then $f$ is SNP \cite{MTY}.

The notion of M-convexity is part of the definition of {\em Lorentzian polynomials}, defined by Br\"and\'en--Huh in \cite{bh20}, as
a common generalization of stable polynomials (a multivariate analogue
of real-rooted polynomials) and volume polynomials in algebraic geometry. 
 A polynomial $f$ in $H_k^n$ with nonnegative coefficients is {\em
    Lorentzian} if and only if (i) its support is M-convex and (ii)
  the Hessian of any of its partial derivatives of order $n-2$ has at
  most one positive eigenvalue \cite{bh20}.

Lorentzian polynomials are of interest in part because they satisfy both a discrete and continuous type of log-concavity  (see \cite[Section 2.4, Proposition 4.4]{bh20} and Proposition~\ref{thm: log concavity of LP}). Br\"and\'en and Huh used the theory of Lorentzian
polynomials to prove the strongest version of {\em Mason's conjecture} \cite[Theorem 4.14]{bh20}\footnote{The strongest version of Mason's conjecture was also proved independently and simultaneously in \cite{ALGVIII}.}: The numbers $I_k$ of
independent sets of size $k$ in a matroid with $n$ elements form an
{\em ultra log-concave} sequence \cite{mason}.
Huh, Matherne, M\'esz\'aros, and St. Dizier showed in
\cite{HMMStD} that (normalized) Schur functions and certain Schubert polynomials are also Lorentzian.  They also conjectured that a host of other Schur-like polynomials in algebraic combinatorics should be Lorentzian.

Huh showed that the coefficients of chromatic polynomials of
graphs are log-concave \cite{h12}. Because of the advent of Lorentzian
polynomials to study log-concavity of multivariate polynomials in
algebraic combinatorics, it is natural to consider chromatic
symmetric functions $X_G({\bf x})$.

The main conjecture of this paper is that chromatic symmetric functions of Dyck paths are Lorentzian\footnote{Liu and Vinzant (private communication) found a counterexample to this conjecture for $n=8$. See Appendix~\ref{appendix:counterex}.}.

\begin{conjecture}[{Conjecture~\ref{conj:dyckLor}}] \label{conj: csf
    lorentzian}
 Let $d$ be a Dyck path. Then $X_{G(d)}$, restricted to any finite number of variables, is Lorentzian.
\end{conjecture}

We verify Conjecture~\ref{conj: csf lorentzian} in the special case where the indifference graph $G(d)$ is co-bipartite.  Dyck paths of this type are called \emph{abelian} in the literature \cite{HaPre}, and they form an important class of Dyck paths with connections to Lie theory \cite{HaPre} and ($q$) rook theory \cite{StSt,AN,CMP}.

\begin{theorem}[{Theorem~\ref{thm: Lorentzian abelian case}}]\label{thm: abel Lor}
Let $d$ be an abelian Dyck path.  Then $X_{G(d)}$, restricted to any finite number of variables, is Lorentzian.
\end{theorem}

The proof of Theorem~\ref{thm: abel Lor} has interesting connections to rook theory.  Because the monomial expansion of $X_{G(d)}$, for abelian $d$, has coefficients involving rook numbers, a key role in the proof is played by the real-rootedness of the \emph{hit polynomial} of any Ferrers board \cite[Theorem 1]{HOW}, which implies that its coefficients form an ultra log-concave sequence.

For arbitrary Dyck paths $d$, calculations suggest that $X_{G(d)}({\bf x})$ may be {\em stable} (see Conjecture~\ref{conj:dyckstab}), a more
restrictive condition studied by Borcea and Br\"and\'en \cite{BB1,BB2} that
implies the Lorentzian property and is related to real-rootedness.

As a partial result toward Conjecture~\ref{conj: csf lorentzian} in the general case, we note that
our Theorem~\ref{thm:main}~(2) asserts that the support of $X_{G(d)}({\bf x})$
is M-convex since it is a permutahedron, and therefore a generalized
permutahedron.  However, Conjecture~\ref{conj: csf lorentzian} does not extend to the more general class of incomparability graphs of
(3+1)-free posets (see Example~\ref{ex:faillor}), even though
Theorem~\ref{thm:main}~(3) shows their chromatic symmetric functions are
M-convex. See Table~\ref{fig:summary} for a schematic summary of our results and conjectures.

%%%%%%%%%%%%%%%%%%%%%%%%%%%%%%%%%%%%%%%%%%%%%%%%%%%%%%%%%%%%%%%%%%%%%%%%%%%
\subsubsection{M-convexity and claw free graphs}
%%%%%%%%%%%%%%%%%%%%%%%%%%%%%%%%%%%%%%%%%%%%%%%%%%%%%%%%%%%%%%%%%%%%%%%%%%%

Monical conjectured a relation between $s$-positive chromatic symmetric functions and the SNP property.

\begin{conjecture}[Monical \cite{Mphd}]\label{conj:mon}
If $X_G$ is $s$-positive, then $X_G(x_1,\ldots,x_k)$ is SNP for any $k$.
\end{conjecture}

Gasharov showed that for (3+1)-free posets \(P\), $X_{G(P)}$ is \(s\)-positive \cite{G2}. Thus, Theorem~\ref{thm:main} (3) is a partial confirmation of Monical's conjecture. We further investigate Monical's conjecture and Conjecture~\ref{conj:stan} in Section~\ref{snpsection}. We find that the strengthening of Conjecture~\ref{conj:mon} fails if we want $X_G$ to be M-convex, rather than just SNP; see Example~\ref{example:triforce} for a claw-free graph $G$ for which $X_G$ is $s$-positive but is not M-convex.

\subsubsection{Computational complexity of our classes of chromatic symmetric functions}

Inspired by recent work of Adve--Robichaux--Yong \cite{ARY,ARY2}, we use the explicit descriptions of the Newton polytopes in Theorem~\ref{thm:main} to analyze the complexity of computing coefficients of our three classes of chromatic symmetric functions (see Section~\ref{complexitysubsection}).  Throughout this section, we write $X_{G} = \sum_{\alpha} c^G_{\alpha} {\bf x}^{\alpha}$, $X_{G(d)} = \sum_{\alpha} c^d_{\alpha} {\bf x}^{\alpha}$, and $X_{G(P)}=\sum_{\alpha} c^P_{\alpha} {\bf x}^{\alpha}$ for the three classes.

\begin{theorem}\label{thm:complexthm}\
\begin{enumerate}
   
    \item (Proposition~\ref{prop:nonvanishing 3+1 free}) 
    Deciding whether any given coefficient \(c^P_{\alpha}\) is nonzero is in \(\mathsf{P}\).
    \item (Proposition~\ref{prop:complexity-cobipartite}) 
    Determining the value of any given coefficient \(c^G_{\alpha}\) is \(\#\mathsf{P}\)-complete.
\end{enumerate}
\end{theorem}

Theorem~\ref{thm:complexthm} (1) implies that deciding nonvanishing of any given $c_\alpha^G$ or $c_\alpha^d$ also takes polynomial time. Similarly, Theorem~\ref{thm:complexthm} (2) implies that determining the value of any given $c_\alpha^P$ is also $\#\mathsf{P}$-complete. We leave open the interesting question of whether or not determining the coefficients $c_\alpha^d$ is $\#\mathsf{P}$-complete.

\begin{table}
    \centering
    \begin{tabular}{|m{3.5cm}"m{4.5cm}|m{4.5cm}|} \hline  
    \(X_{G(P)}\) for (3+1)-free posets \(P\) which
         & are (2+2)-free (indifference graphs of Dyck paths) & may have 2+2 pattern\\
         \thickhline
    are 3-free (co-bipartite graphs) &  (abelian indifference graphs) is Lorentzian by Theorem~\ref{thm: Lorentzian abelian case} & may not be Lorentzian, see Example~\ref{ex:faillor} \\
    \hline
    may have 3 pattern & is conjecturally Lorentzian by Conjecture~\ref{conj:dyckLor} & is M-convex by Theorem~\ref{thm:main}\\\hline
    \end{tabular}
    \caption{Schematic summary of results and conjectures for chromatic symmetric functions of incomparability graphs of certain families of posets.}
    \label{fig:summary}
\end{table}

\subsection{Outline}
In Section~\ref{sec:background}, we present background material on (chromatic) symmetric functions as well as on various properties of their support.  The main results of Sections~\ref{sec: case co-bipartite graphs},~\ref{sec: case dyck paths}, and~\ref{sec: case 3+1 free posets} are that the Newton polytopes of the chromatic symmetric functions of co-bipartite graphs, indifference graphs of Dyck paths, and incomparability graphs of (3+1)-free posets, respectively, are permutahedra.  A direct consequence is that these chromatic symmetric functions are SNP and moreover M-convex.  We conclude the paper with Sections~\ref{sec:stabandlor} and~\ref{sec: conjecture} which collect a number of examples and conjectures about these classes of chromatic symmetric functions: most notably, we conjecture that chromatic symmetric functions of indifference graphs of Dyck paths are Lorentzian, and we verify the conjecture for abelian Dyck paths.  We also use our description of the Newton polytopes to analyze the complexity of our classes of chromatic symmetric functions and to make a conjecture about the $\zeta$ map from diagonal harmonics (e.g. see \cite[Theorem 3.15]{HagqtCat}) relating two Dyck paths encoding unit interval orders.

%%%%%%%%%%%%%%%%%%%%%%%%%%%%%%%%%%%%%%%%%%%%%%%%%%%%%%%%%%%%%%%%%%%%%%%%%%%
\section{Background} 
\label{sec:background}
%%%%%%%%%%%%%%%%%%%%%%%%%%%%%%%%%%%%%%%%%%%%%%%%%%%%%%%%%%%%%%%%%%%%%%%%%%%

\subsection{Partitions and symmetric functions}

The \emph{dominance order} on the set of partitions of the same size is defined as follows: $\lambda \preceq \mu$ if $\sum_{i=1}^k \lambda_i \leq \sum_{i=1}^k \mu_i$ for all \(k\). Similarly, we use the dominance order for compositions of the same size: $\gamma \preceq \beta$ if $\sum_{i=1}^k \gamma_i \leq \sum_{i=1}^k \beta_i$ for all \(k\).  %\jm{define majorization order}

Let $\Lambda$ denote the ring of symmetric functions and $\Lambda_k$ be the subring of $\Lambda$ of symmetric polynomials in $k$ variables. Let $m_{\lambda}$ denote the \emph{monomial symmetric functions}
\[
m_{\lambda} = \sum_{\alpha} x_1^{\alpha_1} x_2^{\alpha_2}\cdots,
\]
where the sum is over all permutations $\alpha$ of the vector $\lambda=(\lambda_1,\lambda_2,\ldots)$. Let $s_{\lambda}$ denote the \emph{Schur symmetric functions}
\[
s_{\lambda} = \sum_{\mu} K_{\lambda,\mu} m_{\mu},    
\]
where $K_{\lambda,\mu}$ is the number of {\emph semistandard Young tableaux} (SSYT) of shape $\lambda$ and content $\mu$. Let  $e_{\lambda}$ denote the \emph{elementary symmetric functions}
\[
e_{\lambda} = e_{\lambda_1}\cdots e_{\lambda_{\ell}}, \quad \text{where} \quad e_k = \sum_{i_1<\cdots <i_k} x_{i_1}\cdots x_{i_k}.
\]

Given a basis $g_{\lambda}$ of $\Lambda$, we say that $f$ in $\Lambda$ is \emph{$g$-positive} if in the $g$-expansion of $f=\sum_{\lambda} c_{\lambda} g_{\lambda}$ all the coefficients $c_{\lambda}$ are nonnegative.  For more details on symmetric functions, see \cite[Ch. 7]{EC2}.

\subsection{The saturated Newton polytope (SNP) property}

For a multivariate polynomial  \(p = \sum_{\alpha} c_{\alpha} {\bf x}^{\alpha}\) in  \(\mathbb{R}[x_1, \ldots, x_k]\), the \emph{support} of \(p\), denoted by \(\supp(p)\), is the set
\(\{\alpha \mid c_{\alpha} \neq 0\}\) in \(\mathbb{N}^k\)  of exponents of monomials with nonzero coefficients in \(p\). For a homogeneous polynomial of degree \(n\), the support lies in the $n$th discrete simplex \(\Delta_k^n\), the set of points in \(\mathbb{N}^k\) where the sum of the coordinates is \(n\).

The \emph{Newton polytope} of $p$ is the convex hull of the exponents in the support of $p$; that is,
\[
\Newton(p) = \conv(\alpha \mid \alpha \in \supp(p)) \subset \mathbb{R}^k.
\]

Given a partition $\lambda$ of length $\ell$ and a nonnegative integer $k \geq \ell$, let $\mathcal{P}^{(k)}_{\lambda}$ be the convex hull of permutations of $(\lambda_1,\ldots,\lambda_{\ell},0,\ldots,0)$ in $\mathbb{R}^k$. For a nonnegative integer $k<\ell$, let $\mathcal{P}^{(k)}_{\lambda}$  be $\mathcal{P}^{(\ell)}_{\lambda} \cap \mathbb{R}^k$.

A polynomial \(p \in \mathbb{R}[x_1, \ldots, x_k]\) has \emph{saturated Newton polytope} (``is SNP'') if $\supp(p)=\Newton(p) \cap \mathbb{Z}^k$. That is, \(p\) is SNP if its support coincides with the lattice points of its Newton polytope. This property was defined in \cite{MTY} and studied for polynomials in algebraic combinatorics like Schur functions and \emph{Stanley symmetric functions}, and was conjectured and settled for {\em Schubert} and {\em (double) Schubert polynomials} in \cite{FMStD} and \cite{CC-RMM}, respectively. For example, by Rado's theorem \cite{Rado}, a Schur polynomial $s_{\lambda}(x_1,\ldots,x_k)$ is SNP and its Newton polytope is $\mathcal{P}^{(k)}_{\lambda}$.

A subset $I \subset \mathbb{Z}^k$ is M-convex if for any $i$ in $[k]$ and any $\alpha$ and $\beta$ in $I$ satisfying $\alpha_i>\beta_i$, there is an index $j$ in $[n]$ such that
\[
\alpha_j <\beta_j \quad \text{and} \quad \alpha-e_i+e_j \in I \quad \text{and}\quad \beta -e_j+e_i \in I.
\]
The convex hull of an M-convex set is a \emph{generalized permutahedron} \cite{PosGP}, and the set of lattice points in an integral generalized permutahedron forms an M-convex set \cite[Theorem 1.9]{M03}. If a homogeneous polynomial has M-convex support, then it is SNP, but the converse does not hold (see Example~\ref{example:triforce}).

 We summarize this discussion for the example of Schur polynomials in the theorem below.

\begin{theorem}[Rado \cite{Rado}] \label{thm: Schurs and P_lambda}
The Schur function $s_{\lambda}(x_1, \ldots, x_k)$ is SNP and its Newton polytope  is the permutahedron $\mathcal{P}^{(k)}_{\lambda}$.  In particular, the support of $s_{\lambda}(x_1, \ldots, x_k)$ is M-convex.
\end{theorem}

\subsection{Indifference graphs of Dyck paths and incomparability graphs}

A \emph{Dyck path} $d$ of length \(n\) is a lattice path from $(0,0)$ to $(n,n)$ with north steps $\sfn = (0,1)$ and east steps $\sfe = (1,0)$ that does not go below the diagonal $y=x$. The \emph{bounce path} of $d$ is the path obtained by starting at $(0,0)$, traveling north along $d$ until a $(1,0)$ step of $d$, and then turning east until the diagonal, then turning north until a $(1,0)$ step of $d$, and then again turning east until the diagonal, continuing this process until arriving at $(n,n)$ \cite[Definition 3.1]{HagqtCat}.  The points $(0,0), (i_1,j_1),\ldots, (i_b,j_b)=(n,n)$ where the bounce path hits the diagonal are called \emph{bounce points}. The {\em area sequence} of $d$ is the tuple of nonnegative integers $(a_1,\ldots,a_n)$ where $a_i$ is the number of squares in row $i$ between the path and the diagonal.

Given a Dyck path $d$ of length $n$, let \(G(d)\) be the \emph{indifference graph} of the Dyck path: the graph where the vertices are \([n]\) and there is an edge between \(i\) and \(j\), with \(i < j\), if the square in column \(i\) and row \(n+1-j\) is between the path and the diagonal. Note that here we use matrix coordinates for the cells of the diagram, i.e. row numbers increase down the diagram. Given a Dyck path $d$ of length $n$, the associated \emph{Hessenberg function} $h_d\colon [n] \rightarrow [n]$ is defined by setting $h_d(i)$ to be the number of squares in column $i$ below $d$.  These functions are characterized as follows:  $h_d(i)\geq i$ for all $i$ in $[n]$, and $h_d(i+1)\geq h_d(i)$ for all $i$ in $[n-1]$. Dyck paths whose indifference graphs are co-bipartite are called {\em abelian} \cite{HaPre}.

The \emph{incomparability graph}  \(G(P)\) of a poset $P$ is the graph formed by taking the elements of \(P\) as vertices, and putting an edge between \(i\) and \(j\) if \(i\) and \(j\) are incomparable in $P$.

\begin{remark} \label{indiff graph to incomp graph}
The indifference graph \(G(d)\) of a Dyck path $d$ with associated Hessenberg function $h_d$ is the incomparability graph  of the poset $P$ on  \([n]\) with relations \(i < j\) whenever \(h_d(i) < j\).
\end{remark}

\begin{example} \label{ex: indifference graph and poset}
The Dyck path $d=\sfn\sfn\sfn\sfe\sfe\sfn\sfn\sfe\sfe\sfe$, together with its indifference graph $G(d)$ and associated poset $P$, is illustrated in  Figure~\ref{fig:ex_path_graph}. Its Hessenberg function is $h_d=(h_d(1),\ldots,h_d(5))=(3,3,5,5,5)$ and the poset $P$ with elements $[5]$ has cover relations $1<4,1<5,2<4,2<5$. 
\end{example}

\begin{figure}
    \centering
    \includegraphics{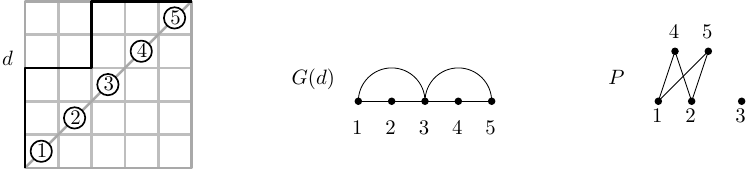}
    \caption{A Dyck path $d$ encoded by the Hessenberg function $h_d=(3,3,5,5,5)$ and its indifference graph $G(d)$ which is an incomparability graph of a unit interval order poset $P$.}
    \label{fig:ex_path_graph}
\end{figure}

\begin{definition} \label{prop: indifference graphs are unit interval orders}
A poset is \emph{(m+n)-free} if there are no two disjoint chains \(a_1 < \cdots < a_m\) and \(b_1 < \cdots < b_n\) in the poset such that every \(a_i\) is incomparable to every \(b_j\).
\end{definition}
\begin{proposition}[{e.g., see \cite[Section 2.1]{AP}}]
Indifference graphs of Dyck paths of length $n$ are exactly the incomparability graphs of (3+1)- and (2+2)-free posets of $n$ elements.
\end{proposition}

\subsection{Chromatic symmetric functions}

For \(G\) a graph, \(f\colon V(G) \to \mathbb{N}\) is \emph{proper} if the inverse image of every number (called a color) is an independent subset of the graph's vertices, that is, a subset of the vertices where no two  are adjacent.

The \emph{chromatic symmetric function} for \(G\), defined in \cite{StChrom1}, is the infinite sum

\[
X_G({\bf x}) = \sum_{\overset{f\colon V(G) \to \mathbb{N}}{f \text{ proper}}} {\bf x}^f,
\]
where the sum is over all proper colorings of $G$, and the monomial \({\bf x}^f\) is notation for
\[
{\bf x}^f = \prod_{v \in G} x_{f(v)} = x_1^{|f^{-1}(1)|}x_2^{|f^{-1}(2)|}\cdots.
\]
We call the vertex sets $f^{-1}(i)$ \emph{color classes}. When we restrict it to \(k\) variables (as though the rest were zero),
\[
X_G(x_1, \ldots, x_k) = \sum_{\overset{f\colon V(G) \to [k]}{f \text{ proper}}} {\bf x}^f.
\]

For a coloring  \(f\colon V(G) \to [k]\), define the \emph{weight} of \(f\) to be
\[
\wt(f) = (|f^{-1}(1)|, |f^{-1}(2)|, \ldots, |f^{-1}(k)|) \in \mathbb{N}^k.
\]
Thus, the support of \(X_G(x_1, \ldots, x_k)\) is the set
\[
\{ \wt(f) \mid f\colon V(G) \to [k] \text{ is proper}\}.
\]

Since \(X_G\) is a symmetric function, if \(\alpha \in \supp(X_G)\) then any permutation of \(\alpha\) is also in \(\supp(X_G)\).  Throughout, we will say that a graph $G$ is \emph{$g$-positive} if its chromatic symmetric function $X_G$ is $g$-positive.

\subsection{Chromatic symmetric functions of co-bipartite graphs} \label{sec: cobipartite graphs and rooks}

Stanley and Stembridge \cite{StSt} related $X_G$ of co-bipartite graphs $G$ with {\em rook theory}. Given a \emph{board} $B\subset [n_1] \times [n_2]$, let $r_k=r_k(B)$ be the number of placements of $k$ non-attacking rooks on $B$ (e.g. see \cite{kaplansky1946problem}). Given such a co-bipartite graph $G$, i.e. a complement of a bipartite graph, with vertex set $\{1,\ldots,n_1\} \cup \{n_1+1,\ldots,n_1+n_2\}$, we associate to it a board $B\subset [n_1]\times [n_2]$ with a cell $(i,j)$, in matrix coordinates, for each edge $(i,n_1+j)$ {\em not} in $G$. In the case of abelian Dyck paths $d$, the graph $G(d)$ is encoded by a {\em Ferrers board} $B_{\mu} \subset [n_1]\times [n_2]$ of a partition $\mu=(\mu_1,\ldots,\mu_{\ell})$. The board \(B_\mu\) has a cell \((i,j)\) if \(j \leq \mu_i\), i.e, \(B_\mu\) consists of a justified collection of $\mu_i$ boxes in the $i$th row for $i=1,\ldots,\ell$ (see Example~\ref{ex: cobipartite}).

\begin{lemma}[{Stanley--Stembridge \cite[Remark 4.4]{StSt}}] \label{lem:StanleyStembridgeCobipartite}
Let \(G\) be a co-bipartite graph with vertex set $\{1,\ldots,n_1\} \cup \{n_1+1,\ldots,n_1+n_2\}$, and let $B$ be the board associated to $G$. We have
\begin{equation} \label{eq:rel_rooks}
X_G = \sum_{i} i! \cdot (n_1+n_2-2i)!\cdot r_i(B) \cdot m_{2^i1^{n_1+n_2-2i}}.
\end{equation} 
\end{lemma}

\begin{example} \label{ex: cobipartite}
Continuing with Example~\ref{ex: indifference graph and poset}, the Dyck path  $d=\sfn\sfn\sfn\sfe\sfe\sfn\sfn\sfe\sfe\sfe$ in Figure~\ref{fig:ex_path_graph} is abelian since $G(d)$ is co-bipartite with vertices $\{1,2\}\cup \{3,4,5\}$. We associate to $G(d)$ the Ferrers board $B_{22} \subset [2] \times [3]$. For this board we have $r_0=1,r_1=4,r_2=2$, and so by \eqref{eq:rel_rooks} we have that  $X_{G(d)} = 120 m_{11111} + 24 m_{2111} + 4m_{221}$.
\end{example}

\subsection{Stanley's nice property of chromatic symmetric functions}\label{sec:nice}

A graph $G$ is \emph{nice} if whenever $\lambda$ is in $\supp(X_G)$ and $\mu \preceq \lambda$, then $\mu$ is in $\supp(X_G)$. Stanley introduced this notion in \cite{StChrom2} and deduced the following properties.

\begin{proposition}[{Stanley \cite[Proposition 1.5]{StChrom2}}]
\label{prop:schur-nice}
If $G$ is $s$-positive, then $G$ is nice.
\end{proposition}

To state the next result, we need the following definition. A graph $G$ is \emph{claw-free} if it does not have the claw graph $K_{1,3}$ as an induced subgraph.

\begin{proposition}[{Stanley \cite[Proposition 1.6]{StChrom2}}]
\label{prop:claw-nice}
A graph $G$ and all of its induced subgraphs are nice if and only if $G$ is claw-free.
\end{proposition}

The following families of graphs are known to have $s$-positive chromatic symmetric functions:

\begin{compactitem}
\item[(i)] co-bipartite graphs \cite[Corollary 3.6]{StChrom1} (or incomparability graphs of 3-free posets),
\item[(ii)] indifference graphs of Dyck paths, i.e. incomparability graphs of unit interval orders (or (3+1)- and (2+2)-free posets)  \cite{StSt}, and
\item[(iii)] incomparability graphs of (3+1)-free posets \cite{G2}.
\end{compactitem}
Note that families (i) and (ii) are contained in (iii).

%%%%%%%%%%%%%%%%%%%%%%%%%%%%%%%%%%%%%%%%%%%%%%%%%%%%%%%%%%%%%%%%%%%%%%%%%%%
\section{Chromatic symmetric functions of co-bipartite graphs} 
\label{sec: case co-bipartite graphs}
%%%%%%%%%%%%%%%%%%%%%%%%%%%%%%%%%%%%%%%%%%%%%%%%%%%%%%%%%%%%%%%%%%%%%%%%%%%

Let \(G\) be a co-bipartite graph with $n$ vertices, not necessarily an indifference graph of a Dyck path. Stanley \cite[Corollary 3.6]{StChrom1} showed that \(X_G\) is $e$-positive and thus $s$-positive.

By Lemma~\ref{lem:StanleyStembridgeCobipartite} the expansion of \(X_G\) in the monomial basis is
\begin{equation} \label{eq: monomial expansion co-bipartite}
X_G = \sum_{i = 0}^{\lfloor n/2\rfloor } c^G_{2^i1^{n-2i}} m_{2^i1^{n-2i}},
\end{equation}
with some coefficients \(c^G_{2^i1^{n-2i}}\) possibly \(0\). Let $\lambda(G)=2^j1^{n-2j}$, where $j$ is maximal such that $c^G_{2^j1^{n-2j}}\neq 0$. Next, we show that  \(X_G\) is SNP. 

\begin{proposition}
\label{co-bipartite-m-convex}
If $G$ is a graph with $n$ vertices and its complement \(\overline{G}\) is bipartite, then \(X_G(x_1,\ldots,x_k)\) is SNP and its Newton polytope is $\mathcal{P}^{(k)}_{\lambda(G)}$.
\end{proposition}
\begin{proof}
Since \(G\) is a union of two cliques and edges between the cliques, \(G\) is claw-free, and so is nice via Proposition~\ref{prop:claw-nice}.
 The partitions \(2^i1^{n-2i}\) appearing in the monomial expansion in \eqref{eq: monomial expansion co-bipartite} are totally ordered by dominance, so there is a unique maximal \(\lambda\) such that \(c_\lambda \neq 0\), and this is exactly $\lambda(G)$. Since \(G\) is nice, the support of \(X_G\) is the same as the support of \(s_{\lambda(G)}\), which by Theorem~\ref{thm: Schurs and P_lambda} is $\mathcal{P}^{(k)}_{\lambda(G)}$.
\end{proof}

\begin{example}
Consider the co-bipartite graph $G$ with vertices $\{1,2\}\cup\{3,4\}$ and edges $\{(1,2),(3,4),(2,3)\}$. Its chromatic symmetric function is $X_G = 24m_{1111} +6m_{211} + 2m_{22}$, and $\Newton(X_{G(d)}(x_1,x_2,x_3,x_4))=\mathcal{P}^{(4)}_{22}$. 
\end{example}

%%%%%%%%%%%%%%%%%%%%%%%%%%%%%%%%%%%%%%%%%%%%%%%%%%%%%%%%%%%%%%%%%%%%%%%%%%%
\section{Chromatic symmetric functions of Dyck paths} 
\label{sec: case dyck paths}
%%%%%%%%%%%%%%%%%%%%%%%%%%%%%%%%%%%%%%%%%%%%%%%%%%%%%%%%%%%%%%%%%%%%%%%%%%%

Recall that any graph can be colored with a \emph{greedy coloring} relative to a fixed ordering on the vertices. Given a Dyck path $d$, let $\lambda^{\gr}(d)$ be the weight \(\wt(\gr)\) of the greedy coloring on the indifference graph $G(d)$.

\begin{theorem}\label{thm:p_lambda}
Let $d$ be a Dyck path. Then $X_{G(d)}(x_1,\ldots,x_k)$ is SNP and its Newton polytope is $\mathcal{P}^{(k)}_{\lambda^{\gr}(d)}$.
\end{theorem}

\begin{example} \label{ex: running ex1}
For the Dyck path $d=\sfn\sfn\sfn\sfe\sfe\sfn\sfn\sfe\sfe\sfe$ in Figure~\ref{fig:ex_path_graph}, we have that $\lambda^{\gr}(d)=(2,2,1)$,  $X_{G(d)} = 120 m_{11111} + 24 m_{2111} + 4m_{221}=36s_{11111} +  16s_{2111} + 4s_{221}$, and $\Newton(X_{G(d)}(x_1,\ldots,x_k))=\mathcal{P}^{(k)}_{221}$ (see Figure~\ref{fig:ex newton polytope}).
\end{example}

\begin{corollary} \label{cor: M-convex for Dyck paths}
Let \(d\) be a Dyck path. Then \(\supp(X_{G(d)}(x_1,\ldots,x_k))\) is M-convex.
\end{corollary}

\begin{proof}
The result follows by Theorem~\ref{thm:p_lambda} and the fact that the support of a homogeneous polynomial $p$ with nonnegative coefficients is M-convex if and only if $p$ is SNP and its Newton polytope is a generalized permutahedron.
\end{proof}

\begin{figure} 
\begin{center}
\raisebox{10pt}{\includegraphics[scale=0.3]{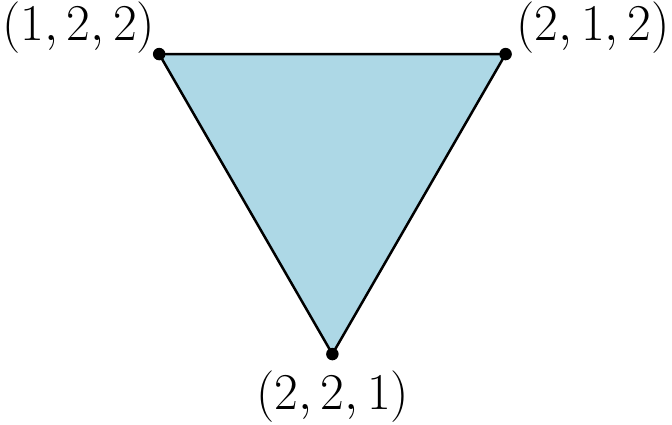}}
\qquad \qquad
\includegraphics[scale=0.3]{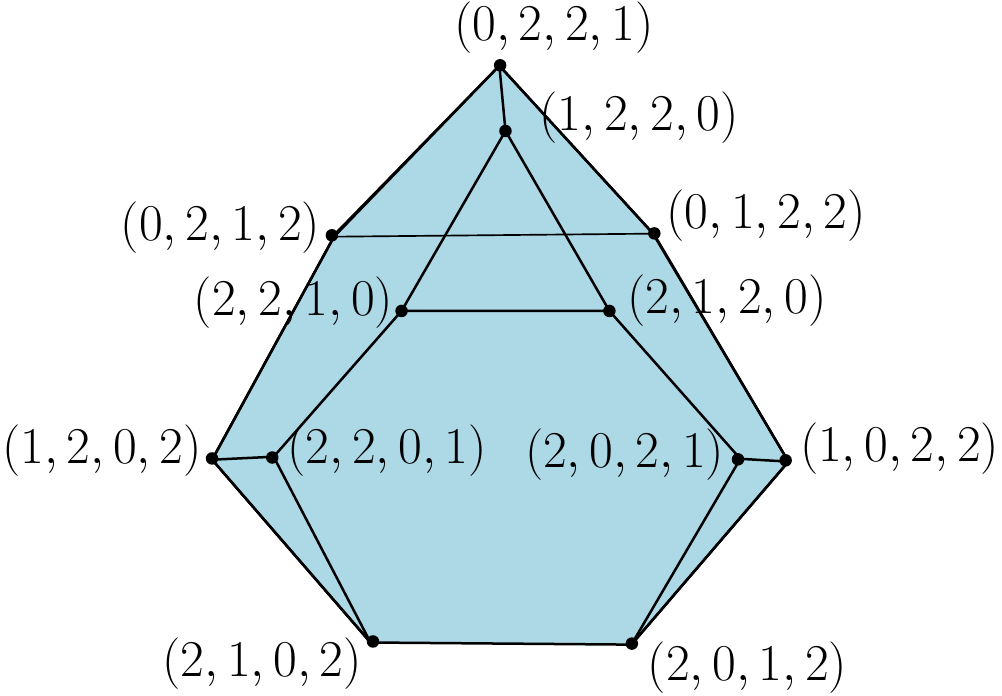}
\end{center}
\caption{The Newton polytopes of $X_{G(d)}(x_1,x_2,x_3)$ and $X_{G(d)}(x_1,x_2,x_3,x_4)$ for $d=\sfn \sfn \sfn \sfe \sfe \sfn \sfn \sfe \sfe \sfe$ are the permutahedra $\mathcal{P}_{221}^{(3)}$ and $\mathcal{P}^{(4)}_{221}$, respectively.}\label{fig:ex newton polytope}
\end{figure}

\begin{proof}[Proof of Theorem~\ref{thm:p_lambda}]
Since $G(d)$ is claw-free,  Proposition~\ref{prop:claw-nice} asserts that $G(d)$ is nice. This means that if a partition \(\lambda\) is in \(\supp(X_{G(d)})\), then \[\supp(s_{\lambda}(x_1, \ldots, x_k)) = \mathcal{P}^{(k)}_{\lambda}  \subset \supp(X_{G(d)}(x_1,\ldots,x_k)).\] In particular \(\mathcal{P}^{(k)}_{\lambda^{\gr}(d)}  \subset \supp(X_{G(d)}(x_1,\ldots,x_k))\). By  Lemma~\ref{lemma greedy is max} below, the reverse inclusion holds and the result follows.
\end{proof}

\begin{remark}
In \cite[Proposition 2.5 III]{MTY}, Monical--Tokcan--Yong generalize the strategy we use here as a general lemma to give a criterion for a symmetric function to be SNP and have a Newton polytope which is a permutahedron.
\end{remark}

\subsection{Greedy coloring on Dyck paths}

For an indifference graph \(G(d)\) on \([n]\), we can describe the greedy coloring algorithm using the Dyck path \(d\). 

\begin{definition}[bounce path coloring]
Let \(G(d)\) be the indifference graph of a Dyck path \(d\), and let \(h_d\) be the associated Hessenberg function. The {\em bounce path coloring} of $G(d)$ is defined as follows. For each color \(i\) in order, select the vertices which will be colored \(i\) by the following procedure:
Start at the first uncolored vertex \(j\), and color it \(i\). Set \(j\) to the first uncolored vertex greater than \(h_d(j)\), color it \(i\), and repeat until the end of the graph is reached.
\end{definition}

\begin{proposition}
Let $G(d)$ be the indifference graph of a Dyck path $d$.  Then the bounce path coloring is the greedy coloring of $G(d)$.
\end{proposition}
\begin{proof}
In a greedy coloring, the set of vertices colored $1$ can be found by iteratively building a list, starting with the first vertex, and adding any vertex that is not adjacent to any vertex in the list. Not considering any of the vertices on this list, we can repeat the process to find the vertices colored $2$, and so on. In an indifference graph $G(d)$, the process can be simplified: if the list of vertices for color $i$ during the iteration is  $S_i=\{v_1,\ldots,v_k\}$, then a later vertex  $v$ is not adjacent to any $v_i$ if and only if $v$ is not adjacent to $v_k$. The latter is true if and only if $h_d(v_k)< v$. Thus, the final list $S_i$ is determined by the bounce path coloring construction.
\end{proof}

\begin{remark} \label{rem: description greedy coloring with bounce path}
For a vertex $j$, the vertex $h_d(j)$ is the next vertex hit by a bounce path on $d$ starting at $j$. Thus, the greedy coloring defined above can be viewed in the Dyck path $d$ as follows. Starting at the bottom left corner of $d$, do a bounce path and color the vertices the path visits (when it bounces off the diagonal) with color $1$. Then, start another bounce path before the first uncolored vertex. If the path visits a colored vertex on the diagonal, then the path follows the diagonal until it bounces off before the next uncolored vertex. Color the vertices visited when the path bounces off the diagonal with color $2$, and so on.

Note that \(\lambda^{\gr}(d)_1\) is the number of bounce points of the bounce path of $d$, excluding $(n,n)$.
\end{remark}

\begin{remark}\label{rem:gers}
In the process of determining the closure order on nilpotent orbits in type $A$, Gerstenhaber gave an algorithm to determine $\lambda^{\gr}(d)$ \cite{Gerstenhaber}.  We point to \cite[Section 6]{FennSommers} for a modern description of the algorithm and further properties of $\lambda^{\gr}(d)$. 
\end{remark}

\begin{figure}
\includegraphics[scale=0.8]{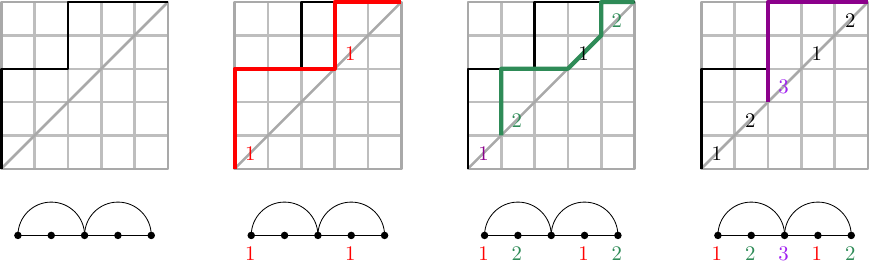}
\caption{Description of bounce path algorithm to determine the greedy coloring  weight $(2,2,1)$.}\label{fig: bouncepathalg}
\end{figure}

\begin{proposition}
The greedy coloring weight $\lambda^{\gr}(d)$ is a partition; i.e. it is a sorted weight vector. 
\end{proposition}

\begin{proof}
Consider the bounce points of each color's bounce paths. Since the bounce path for color \(i\) starts before the bounce path for color \(i+1\), the first bounce point for color \(i\) is before the first bounce point for color \(i+1\), and so the second bounce point for color \(i\) is before the second bounce point for color \(i+1\), and so on. Thus the total number of bounce points for color \(i\) is at least the total number of bounce points for color \(i+1\), and therefore \(\lambda^{\gr}(d)_i \geq \lambda^{\gr}(d)_{i+1}\).
\end{proof}

\begin{example}
Continuing with Example~\ref{ex: running ex1}, the bounce path greedy coloring of $G(d)$ for $d=\sfn\sfn\sfn\sfe\sfe\sfn\sfn\sfe\sfe\sfe$  is illustrated in  Figure~\ref{fig: bouncepathalg}. The greedy coloring weight is $\lambda^{\gr}(d)=(2,2,1)$.
\end{example}

\subsection{Greedy coloring gives dominating partition}

In this section, we show that the greedy coloring weight dominates the weight of any other coloring.

\begin{lemma} \label{lemma greedy is max}
Given a Dyck path $d$, let $X_{G(d)} = \sum_{\lambda} c^d_{\lambda} m_{\lambda}$.  If $c^d_{\lambda} \neq 0$ for some $\lambda$, then $\lambda \preceq \lambda^{\gr}(d)$ in dominance order.
\end{lemma}

\begin{remark}\label{rem:coloringlemma}
While preparing the current paper, the authors learned that this result, and a similar proof, were also known to Tim Chow in an
unpublished note \cite{TC} (where he calls the greedy coloring the {\em first-fit coloring}).
\end{remark}

\begin{proof}[Proof of Lemma~\ref{lemma greedy is max}]
For each \(k\) in $[n]$, it suffices to show that
\begin{equation} \label{eq: wts}
\sum_{i=1}^k \lambda^{\gr}(d)_i = \max_{f \text{ proper}} \sum_{i=1}^k \wt(f)_i.
\end{equation}

We say a proper coloring $f$ is \emph{$k$-maximal} if  $\sum_{i=1}^k \wt(f)_i$ is maximal among all proper colorings. Our strategy is as follows: we fix $k$ in $[n]$ and show by induction on $j\geq 0$ that for all \(j\) in \([n]\) there exists a \(k\)-maximal coloring \(f\) such that

\begin{itemize}
    \item[(*)] $f(i)=\gr(i)$ for all vertices $i$ in $[j]$ such that $f(i)$ is in $[k]$.
\end{itemize}
Equation \eqref{eq: wts} then follows from (*) since when \(j = n\), we see that the greedy coloring must also be \(k\)-maximal. Since this holds for all $k$ in $[n]$, the greedy coloring is maximal in dominance order.

The base case $j=0$ is true since $k$-maximal colorings exist and condition (*) is vacuously true. Next, suppose that we have a $k$-maximal coloring \(f\) which satisfies condition (*) for some $j\geq 0$. 

Consider the vertex
\(j+1\). If \(f(j+1)\) is not in \([k]\) or \(f(j+1) = \gr(j+1)\), then \(f\) also satisfies (*) for \(j+1\), so we are done.

Otherwise if $f(j+1)$ is in $[k]$ and $f(j+1)\neq \gr(j+1)$, we claim that $\gr(j+1)$ is also in $[k]$. To see this, it is enough to show that 
\[
  \gr(j+1) < f(j+1).
\]
This inequality holds because the greedy coloring will assign the first available color to the vertex \(j+1\), and since $\gr$ agrees with $f$ on the first $j$ vertices, the first available color \(\gr(j+1)\) is at most \(f(j+1)\). 

Let \(c=\gr(j+1)\) and \(d=f(j+1)\). We will create a new \(k\)-maximal coloring \(f'\) such that condition (*) is satisfied for \(j+1\), by swapping the colors \(c\) and \(d\) in \(f\) after position \(j\).

  % Swap tails
  Concretely, let \(f'\) be given by
  \[
    f'(i) =
      \begin{cases}
        c & \text{ if } i \geq j+1, f(i) = d \\
        d & \text{ if } i \geq j+1, f(i) = c \\
        f(i) & \text{otherwise}.
      \end{cases}
  \]
  
  Note that \(f'\) is still \(k\)-maximal because we swapped one color in \([k]\) for another. Since \(f'(j+1) = \gr(j+1)\) and \(f'(i)=f(i)=\gr(i)\) for $i$ in $[j]$, it follows that \(f'\) satisfies condition (*) for \(j+1\). Thus it remains to show that \(f'\) is proper.
  
  Only the colors \(c\) and \(d\) have changed from \(f\) to \(f'\), so
  let \(v\) be a vertex prior to \(j+1\) which is colored either \(c\) or \(d\) in \(f'\). Since \(f(j+1) = d\) and \(\gr(j+1) = c\), and both of those colorings are proper, the value of the associated Hessenberg function \(h_d(v) < j+1\). This means \(v\) is not adjacent to any vertex \(v'\) after \(j+1\), and so both \(f'^{-1}(c)\) and \(f'^{-1}(d)\) are still independent sets, as desired.
\end{proof}

As a corollary we obtain a similar result as Lemma~\ref{lemma greedy is max} in the Schur basis.

\begin{corollary} \label{thm: max Schur expansion}
Given a Dyck path $d$, let $X_{G(d)} = \sum_{\lambda} f^{d}_{\lambda} s_{\lambda}$. If $f^{d}_{\mu}\neq 0$ for some \(\mu\), then $\mu \preceq \lambda^{\gr}(d)$. 
\end{corollary}

\begin{proof}
Since $X_{G(d)}$ is $s$-positive, if $f^d_{\lambda}> 0$ for some $\lambda$ then the coefficient $c^d_{\lambda}$ in the monomial basis is also positive, and the result follows by Lemma~\ref{lemma greedy is max}.
\end{proof}

Lastly, given any partition we can find a Dyck path whose chromatic symmetric function $X_{G(d)}$ has as Newton polytope the permutahedron associated to $\lambda$.

\begin{proposition}
Given a partition $\lambda$, the chromatic symmetric function $X_{G(d)}(x_1,\ldots,x_k)$ for the Dyck path $d= \sfn^{\lambda'_1}\sfe^{\lambda'_1} \cdots \sfn^{\lambda'_m}\sfe^{\lambda'_m}$ where $m=\lambda_1$ has Newton polytope $\mathcal{P}^{(k)}_{\lambda}$.
\end{proposition}

\begin{proof}
The graph \(G(d)\) consists of \(m\) cliques of sizes \(\lambda'_1\) through \(\lambda'_m\). The greedy coloring will color the \(i\)th clique with the colors \(\{1, \ldots, \lambda'_i\}\). In this coloring, the color \(j\) is used \(\#\{i \mid \lambda'_i \geq j\} = \lambda_j\) times, thus $\lambda^{\gr}(d)=\lambda$. The result then follows by Theorem~\ref{thm:p_lambda}.
\end{proof}

%%%%%%%%%%%%%%%%%%%%%%%%%%%%%%%%%%%%%%%%%%%%%%%%%%%%%%%%%%%%%%%%%%%%%%%%%%%
\section{Chromatic symmetric functions of (3+1)-free posets} 
\label{sec: case 3+1 free posets}
%%%%%%%%%%%%%%%%%%%%%%%%%%%%%%%%%%%%%%%%%%%%%%%%%%%%%%%%%%%%%%%%%%%%%%%%%%%

\subsection{Structure of (3+1)-free posets} \label{subsec: structure 3+1 free}

The structure, enumeration, and asymptotics of (3+1)-free posets were studied by Lewis--Zhang \cite{LZ} for the labeled case and Guay-Paquet--Morales--Rowland \cite{GPMR} and Guay-Paquet \cite{MGP} for the unlabeled case. We will use results from the unlabeled case using the notation in  \cite{MGP}.

A \emph{part listing} is an ordered list $L$ of parts that are arranged on nonnegative integer levels. Each part is either a vertex at a given level or a bicolored graph with color classes  arranged as vertices on consecutive levels. We can view a part listing  as a word in the alphabet 
\[
\Sigma =\{ v_i \mid i \geq 0\} \cup \{b_{i,i+1}(H) \mid i \geq 0, H \text{ bicolored graph}\},
\]
where $v_i$ represents a vertex on level $i$, and  $b_{i,i+1}(H)$ represents a graph $H$ on levels $i$ and $i+1$ (see Figure~\ref{fig: part listing 3p1 free poset}).

Given a part listing $L$, we associate a poset $P$ on the vertices of $L$ as follows. Given vertices $x$ and $y$, we have that $x<y$ if
\begin{itemize}
    \item[(i)] $x$ and $y$ are, respectively, at levels $i$ and $j$ with $j-i\geq 2$, 
    \item[(ii)] $x$ is one level below $y$ and the part containing $x$ appears before the part containing $y$ in $L$,
    \item[(iii)] $x$ is one level below $y$ and they are joined by an edge of a bicolored graph $H$.
\end{itemize}

\begin{example} \label{ex:part listing}
The part listing $L$ in Figure~\ref{fig: part listing 3p1 free poset} is given by the word $v_0v_1v_2v_2v_0 b_{0,1}(H)$ where $H$ is the bicolored graph with edges $\{(h,d),(i,d),(i,e)\}$. The associated nine element poset $P$ is given in Figure~\ref{fig: 3p1 free poset}.
\end{example}

\begin{theorem}[{\cite[Propositions 2.4, 2.5]{MGP}}] \;
\begin{itemize}
    \item[(i)] Given a (3+1)-free poset $P$, there exists a part listing $L$ whose associated poset is $P$. 
    \item[(ii)] Given a part listing $L$, the associated poset $P$ is (3+1)-free.
\end{itemize}
Moreover, if the part listings $L$ in (i) and (ii) have no parts $b_{i,i+1}(H)$, then the associated poset $P$ is (3+1)- and (2+2)-free (i.e. a unit interval order).
\end{theorem}

Several part listings can correspond to the same (3+1)-free poset. For instance, in Example~\ref{ex:part listing} the same poset as for the part listing $L$ can be obtained from the part listing $L'=b_{1,2}(H)v_0v_1v_2v_2v_0$. There are certain commutation and {\em circulation} relations on the words in $\Sigma$ of listings that yield the same poset (see \cite[Section 3.3]{GPMR} and \cite[Section 2]{MGP}). 

From  \cite[Proposition 3.11]{GPMR}, we can pick a unique part listing representative of a (3+1)-free poset that we call a {\em canonical part listing}.  Moreover, by \cite[Remark 3.2]{GPMR}, such a canonical part listing corresponds to a (3+1)- and (2+2)-free poset if and only if the canonical part listing has no occurrences of $b_{i,i+1}(H)$. We summarize the characterization of canonical part listings of  (3+1)- and (2+2)-free posets in the following result implicit in \cite{GPMR}.\footnote{In \cite{GPMR} the authors use the letter $c_i$ corresponding to {\em clones} that correspond to consecutive copies of the letter $v_i$ in the part listing.} 

% called as follows. %Corresponding to every such poset $P$ there is a set of words in the alphabet $\Sigma$ known as the {\em trace of the dependence graph} of $P$. All such words encode $P$ as part listings  
% A listing is \emph{lex-maximal} if it is the lexicographically maximal word among those corresponding to the same poset. We use the following order on the alphabet $\Sigma$:
% \[
% v_0 < b_{0,1} <v_1 < b_{1,2} < v_2 < \cdots 
% \]
% The next result gives a characterization of unique part listings for (3+1)- and (2+2)-free posets

\begin{theorem}[{\cite[Remark 3.2, Proposition 3.11]{GPMR}}] \label{thm: char unit interval orders}
A part listing $v_{a_1}\cdots v_{a_n}$  of an $n$-element (3+1)- and (2+2)-free poset $P$ is the canonical part listing if and only if $a_1=0$ and $a_{i+1}\leq a_i +1$ for $i=1,\ldots, n-1$. 
\end{theorem}

\begin{remark}
In \cite[Sec. 2 and Sec. 3]{GPMR} the canonical part listing is defined as the lexicographically maximal element of a subset of words in the alphabet $\Sigma$ that is called the {\em trace of the dependence graph}, coming from the theory of trace monoids (see \cite[\S 2.3]{tracesbook}).\footnote{A previous version of the current paper incorrectly restated the definition of the canonical part listing from \cite{GPMR}.} The authors in \cite[\S 6]{gelinas2022proof} obtain the canonical part listing as characterized in Theorem~\ref{thm: char unit interval orders} using a well-chosen order on the entire set of words in the alphabet $\Sigma$, circumventing the use of trace monoids. In what follows we use only the characterization in Theorem~\ref{thm: char unit interval orders}.
\end{remark}

\begin{remark} \label{def: varphi area sequence}
Note that the set of tuples of integers ${\bf a} = (a_1,\ldots,a_n)$ satisfying $a_1=0$ and $0\leq a_{i+1}\leq a_i+1$ is a classical interpretation for the Catalan numbers \cite[Exercise 2.80]{Catbook}. Such tuples have the following bijection with Dyck paths: ${\bf a}$ encodes the area sequence of a Dyck path $d'$, or alternatively ${\bf a} \mapsto d'$ where $d'$ is the Dyck path obtained by replacing each $a_i$ by a north step $\sfn$ and $a_i-a_{i+1}+1$ east steps $\sfe$ \cite[Solution 3.80]{Catbook}. 
\end{remark}

\subsection{Guay-Paquet's reduction from (3+1)-free posets to unit interval orders}

In this section,  given a part listing $L$ of a (3+1)-free poset $P=P(L)$, we write $X(L) \coloneq X_{G(P)}$.

For level $i=0,1,\ldots$ and $j=0,1,\ldots,s$, let $U^{(i)}_j$ be the part listing 
\[
U^{(i)}_j \coloneq v_{i+1}^{s-j} v_i^r v_{i+1}^j.
\]
For level $i=0,1,\ldots$ and $j=0,1,\ldots,r$, let $D^{(i)}_j$ be the part listing 
\[
D^{(i)}_j \coloneq v_i^j v_{i+1}^s v_i^{r-j}.
\]
If the context is clear, we omit the level $i$ and denote these part listings by $U_j$ and $D_j$ respectively. 

Given a bicolored graph $H$ with $r$ lower vertices, $s$ upper vertices, and $j=0,\ldots,\min(r,s)$, let $q_j$ be the probability that $H$ and a uniformly random matching $M$ with $\min(r,s)$ edges between the lower and upper vertices have $j$ edges in common.

\begin{theorem}[{\cite[Proposition 4.1~(iv)]{MGP}}] \label{thm: decomposing bipartite}
Let $L$ be the part listing of a (3+1)-free poset containing a bicolored graph $b_{i,i+1}(H)$ with $r$ vertices on level $i$ and $s$ vertices on level $i+1$. Then
\[
X(L) = \sum_{j=0}^{\text{min}(r,s)} q_j X(L_j),
\]
where $L_j$ is the part listing obtained from $L$ by replacing $b_{i,i+1}(H)$ with $U_j$ if $r\geq s$ and with $D_j$ if $r<s$, and \(q_j\) is the probability defined above. 
\end{theorem}

\begin{remark}
The probabilities $q_j$ have an interpretation in terms of rook theory. Given such a bicolored graph $H$ with vertex set $\{1,\ldots,r\} \cup \{r+1,\ldots,r+s\}$, its complement $G=\overline{H}$ is a co-bipartite graph corresponding to a board $B \subset [r]\times [s]$ (see Section~\ref{sec: cobipartite graphs and rooks}). Then $q_j = h_j(B)/|r-s|!$ where $h_j(B)$ is the $j$th \emph{hit number} of $B$, which counts the number of placements of $\min(r,s)$ non-attacking rooks on the rectangular board $[r]\times [s]$ with $j$ rooks in $B$. 
\end{remark}

\subsection{Main result for (3+1)-free posets}

First, we define the greedy weight for colorings of an incomparability graph of a (3+1)-free poset. Given a (3+1)-free poset $P$, the weight $\lambda^{\gr}(P)$ is defined as follows. For a part listing $L$ for $P$:
\begin{itemize}
    \item[(i)] apply Theorem~\ref{thm: decomposing bipartite} to every bicolored graph $b_{i,i+1}(H)$ in the part listing,
    \item[(ii)] for each $b_{i,i+1}(H)$, find the largest $j$ such that $q_j \neq 0$ and replace $L$ by $L_j$.
\end{itemize}
At the end, we obtain a part listing $L'$ with no bipartite graphs and thus representing a (3+1)- and (2+2)-free poset (i.e. a unit interval order). 
By Theorem~\ref{thm: char unit interval orders}, there is a lex-maximal part listing $L''$ for that poset satisfying the property $a_1=0$, $a_{i+1}\leq a_i+1$. Using the greedy coloring in the incomparability graph, which is an indifference graph for some Dyck path $d$, we obtain the weight $\lambda^{\gr}(P)\coloneq\lambda^{\gr}(d)$ (see Section~\ref{sec: listing to Dyck paths}).

\begin{theorem}\label{thm:p_lambda 3+1 free}
Let $G(P)$ be an incomparability graph of a (3+1)-free poset.  Then $X_{G(P)}(x_1,\ldots,x_k)$ is SNP, and its Newton polytope is $\mathcal{P}^{(k)}_{\lambda^{\gr}(P)}$. In particular, \(\lambda^{\gr}(P)\) dominates the weight of any coloring of \(G(P)\).
\end{theorem}

\begin{figure}
    \centering
      \centering
  \begin{subfigure}[b]{0.5\textwidth}
    \centering
     \includegraphics{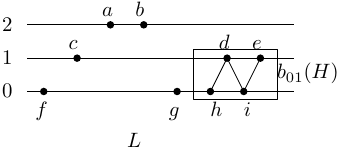}
    \caption{}
    \label{fig: part listing 3p1 free poset}
\end{subfigure}
\begin{subfigure}[b]{0.4\textwidth}
  \centering
  \includegraphics{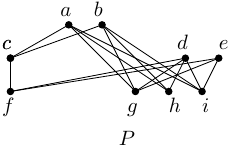}
  \caption{}
  \label{fig: 3p1 free poset}
\end{subfigure}
    \caption{(A) A part listing $L$ and (B) its corresponding (3+1)-free poset $P$.}
\end{figure}

In order to prove Theorem~\ref{thm:p_lambda 3+1 free}, we need the following lemma.

\begin{lemma} \label{lemma proof lemma 3p1 case}
Suppose that \(A\) and \(B\) are part listings, and fix \(r,s\) positive numbers.

Let \(0 \leq j < k \leq s\),
and suppose that we have posets given by part listings as follows:
\[
P_j \coloneq Ab_{i,i+1}(U_j)B \quad \text{and} \quad P_k \coloneq Ab_{i,i+1}(U_k)B.
\]
Then if \(\kappa\) is a weight of a coloring of \(G(P_j)\), there is also a coloring of \(G(P_k)\) with weight \(\kappa\).
The same result holds if we replace $U_i$ with $D_i$.
\end{lemma}

\begin{proof}
It suffices to take \(B\) to be empty by the circulation relation of \cite[Section 2.2]{MGP}. 
Also, it suffices to take \(k = j+1\).
The part listings for \(U_{j}\) and \(U_{j+1}\) are
\[
v_2^{s-j}v_1^rv_2^j \quad \text{and} \quad v_2^{s-j-1}v_1^rv_2^{j+1},
\]
respectively.  Therefore, the poset for
\(
Ab_{i,i+1}(U_{j+1})
\)
is the poset for
\(
Ab_{i,i+1}(U_j)
\)
together with \(r\) additional covering relations coming from moving a \(v_{i+1}\) after all of the \(v_i\)s. Therefore, any coloring of \(G(Ab_{i,i+1}(U_j))\) has a corresponding coloring of \(G(Ab_{i,i+1}(U_{j+1}))\) with the same weight, since adding relations to a poset deletes edges from the incomparability graph, which never makes a proper coloring improper. Therefore, \(G(P_{j+1})\) has a coloring of weight \(\kappa\).
The proof for the case where $U_k$ is replaced with $D_k$ is the same.
\end{proof}

\begin{proof}[Proof of Theorem~\ref{thm:p_lambda 3+1 free}]
Let $L$ be a part listing for \(P \coloneq P(L)\) and recall that for part listings \(F\) we define $X(F) \coloneq X_{G(P(F))}$. The proof is by induction on the number of bicolored graphs in \(L\). If there are none, then \(P\) is also (2+2)-free (i.e. a unit interval order) and thus the graph $G(P)$ is an indifference graph of a Dyck path (see Remark~\ref{indiff graph to incomp graph} and Conjecture~\ref{conj: listing to Dyck path}). The result then follows in this case by Lemma~\ref{lemma greedy is max}. If \(L = Ab_{i,i+1}(H)B\) has at least one bicolored graph, then \(X(L)\) is a convex combination
\[
X(L) = \sum_j q_j X(AU_jB) \quad \text{or} \quad X(L) = \sum_j q_j X(AD_jB).
\]
We proceed with the first case and the argument for the second case is the same. Let \(j'\) be the largest \(j\) such that \(q_j\) is nonzero. The support of \(X(L)\) is the union
\[
\supp X(L)(x_1, \ldots, x_k) = \bigcup_{j=0}^{j'} \supp X(AU_jB)(x_1, \ldots, x_k).
\]
By Lemma~\ref{lemma proof lemma 3p1 case},
\[
\supp X(AU_jB)(x_1, \ldots, x_k) \subset \supp X(AU_{j'}B)(x_1, \ldots, x_k),
\]
so
\[
\supp X(L)(x_1, \ldots, x_k) = \supp X(AU_{j'}B)(x_1, \ldots, x_k).
\]
By the inductive hypothesis, the support of \(X(AU_{j'}B)(x_1, \ldots, x_k)\) is \(\mathcal{P}^{(k)}_{\lambda^{\gr}(AU_{j'}B)}\), which is \(\mathcal{P}^{(k)}_{\lambda^{\gr}(P)}\) by definition.
\end{proof}

\begin{example} \label{big example part listings}
The part listing $L=v_0v_1v_2v_2v_0b_{01}(H)$ in Figure~\ref{fig: part listing 3p1 free poset} has  $X(L)=$
 \begin{multline*}
 362880m_{1^9} + 90720m_{21^7} + 23040m_{2^21^5} + 6048m_{2^31^3} + 1728m_{2^41} + 1440m_{31^6} + 384m_{321^4} + 112m_{32^21^2} + 48m_{32^3}.
 \end{multline*}
The part listings $L_0$, $L_1$, and $L_2$ in Figure~\ref{fig:decomposition part listings} have chromatic symmetric functions
\begin{align*}
    X(L_0) &= 362880m_{1^9} + 75600m_{21^7} + 14880m_{2^21^5} + 2664m_{2^31^3} + 384m_{2^41} \\ & +  1440m_{31^6} + 240m_{321^4} + 32m_{32^21^2},\\
    X(L_1) &= 362880m_{1^9} + 85680m_{21^8} + 20160m_{2^21^5} + 4752m_{2^31^3} + 1152m_{2^41} \\&  + 1440m_{31^6} + 336m_{321^4} + 80m_{32^21^2} + 24m_{32^3},\\
    X(L_2) &= 362880m_{1^9} + 95760m_{21^7} + 25920m_{2^21^5} + 7344m_{2^31^3} + 2304m_{2^41} \\ &  + 1440m_{31^6} + 432m_{321^4} + 144m_{32^21^2} + 72m_{32^3}.
\end{align*}
Next we apply Theorem~\ref{thm: decomposing bipartite}. For the bicolored graph $H$ we have the probabilities $q_0=0$, $q_1=q_2=1/2$, thus 
\[
X(L) = 0\cdot X(L_0) + \frac12 X(L_1) + \frac12 X(L_2).
\]

The part listings $L_0$, $L_1$, and $L_2$ correspond to (3+1)- and (2+2)-free posets. Their respective lex-maximal listings and Hessenberg functions (obtained by inspection, see Conjecture~\ref{conj: listing to Dyck path}) are:

\begin{center}
\begin{tabular}{ccc} \hline 
 & lex-maximal listing & Hessenberg function\\ \hline
$L_0$ & $(0,1,2,2,0,1,1,0,0)$ &  $(4,5,7,7,7,9,9,9,9)$\\
$L_1$  & $(0,1,2,2,0,1,0,0,1)$ & $(4,5,6,6,7,9,9,9,9)$\\
$L_2$ & $(0,1,2,2,0,0,0,1,1)$ & $(4,5,5,5,7,9,9,9,9)$\\ \hline
\end{tabular}
\end{center}
If we perform the greedy algorithm on the incomparability graphs,
%corresponding Dyck paths 
 we obtain the partitions $32211$, $3222$, and $3222$ respectively. Then, by Theorem~\ref{thm:p_lambda 3+1 free},  we have that 
\[
\Newton(X_{G(P)}(x_1,\dots,x_k))=\mathcal{P}^{(k)}_{3222}.
\]
\end{example}

\begin{figure}
    \centering
    \includegraphics{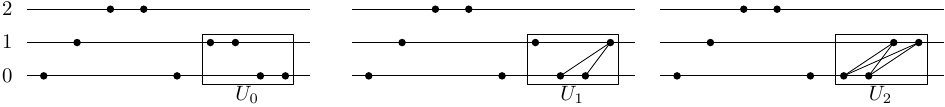}
    \caption{The part listings $L_0$, $L_1$, and $L_2$ in the convex combination of $X(L)$.  The dominant coloring $\kappa_2$ of $X(L_2)$ dominates the respective dominant colorings $\kappa_0$ and $\kappa_1$ of $X(L_0)$ and $X(L_1)$.}
    \label{fig:decomposition part listings}
\end{figure}

\begin{figure}
    \centering
    \includegraphics{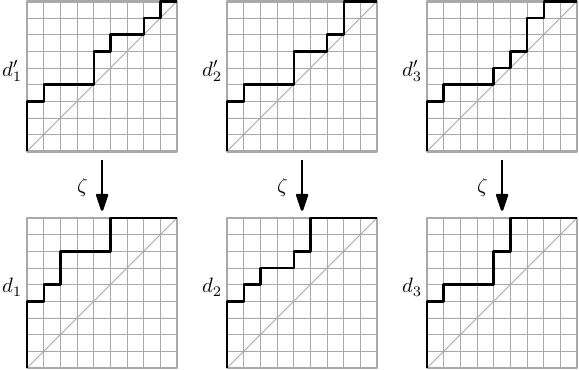}
    \caption{Dyck paths $d'_1,d'_2,d'_3$ corresponding to the lex-maximal listings $L_1,L_2,L_3$ from Example~\ref{big example part listings} and their corresponding Dyck paths $d_1,d_2,d_3$ associated to the incomparability graphs. The conjectured correspondence between these Dyck paths agrees with the $\zeta$ map.}
    \label{fig:big example Dyck path correspondence}
\end{figure}

\section{Stability and the Lorentzian property of $X_{G(d)}$}\label{sec:stabandlor}

\subsection{Main conjectures for all Dyck paths}

Our main results (Proposition~\ref{co-bipartite-m-convex} and Theorems~\ref{thm:p_lambda} and~\ref{thm:p_lambda 3+1 free}) establish that the supports of certain classes of polynomials are M-convex.  The property of M-convexity is often a shadow of a more general property, that of being a Lorentzian polynomial.

Lorentzian polynomials were recently introduced by Br\"and\'en and Huh in \cite{bh20} as a bridge between discrete convex analysis and concavity properties in combinatorics.  Many families of polynomials appearing in algebraic combinatorics are known or conjectured to be Lorentzian: for example (normalized) Schur polynomials, and a variety of other Schur-like polynomials \cite{HMMStD}.

\begin{definition}\label{def:lor}
A homogeneous polynomial $f \in \R[x_1,\ldots,x_k]$ of degree $n$ with nonnegative coefficients is called \emph{Lorentzian} if the following two conditions are satisfied:
\begin{itemize}
    \item $\supp(f)$ is M-convex, and
    \item for all $i_1,i_2,\ldots,i_{n-2} \in [k]$, the associated quadratic form of the quadratic polynomial 
    \[
    \frac{\partial}{\partial x_{i_1}} \circ \cdots \circ \frac{\partial}{\partial x_{i_{n-2}}}(f)
    \]
    has at most one positive eigenvalue. That is, the {\em Hessian} of the quadratic polynomial has at most one positive eigenvalue. 
\end{itemize}
\end{definition}

Note that both conditions in Definition~\ref{def:lor} are ``easy" to check and in particular only require a finite number of checks. An important application of Lorentzian polynomials is that their coefficients form a type of log-concave sequence (and are further log-concave as functions on the positive orthant $\mathbb{R}^k_{>0}$). 

Given a vector $\alpha$ in $\mathbb{N}^k$, let $\alpha! \coloneq \alpha_1!\cdots \alpha_k!$.

\begin{proposition}[{\cite[Theorem 2.30; Proposition 4.4]{bh20}}] \label{thm: log concavity of LP}
Let $f = \sum_{\alpha\in\Delta^n_k} c_{\alpha} \mathbf{x}^\alpha$ be a Lorentzian polynomial.  Then $f$ exhibits the following two types of log-concavity phenomena:
\begin{itemize}
    \item (Continuous) The polynomial $f$ is either identically zero or its logarithm is concave on the positive orthant $\mathbb{R}^k_{>0}$.
    \item (Discrete) The coefficients of $f$ satisfy:
\[(\alpha!)^2 c_\alpha^2 \ge (\alpha+e_i-e_j)! (\alpha-e_i+e_j)!\cdot c_{\alpha+e_i - e_j}c_{\alpha-e_i+e_j} \text{ for all } i,j \text{ in } [k] \text{ and all } \alpha \text{ in } \Delta^n_k,\]
and thus
\[
c_\alpha^2 \ge  c_{\alpha+e_i - e_j}c_{\alpha-e_i+e_j} \text{ for all } i,j \text{ in } [k] \text{ and all } \alpha \text{ in } \Delta^n_k.
\]
\end{itemize}
\end{proposition}

We used SageMath \cite{Sage-Combinat} to check the conditions in Definition~\ref{def:lor}, and verified the following conjecture for all Dyck paths of length $n\leq 7$, with \(k\leq 8\) variables\footnote{Liu and Vinzant (private communication) found a counterexample to this conjecture for $n=8$. See Appendix~\ref{appendix:counterex}.}.

\begin{conjecture}\label{conj:dyckLor}
Let $d$ be a Dyck path. Then $X_{G(d)}$, restricted to any finite number of variables, is Lorentzian.
\end{conjecture}

Theorem~\ref{thm: abel Lor} verifies this conjecture for abelian Dyck paths. Graph colorings have many other interesting log-concavity properties like the following result of Huh on chromatic polynomials of graphs.

\begin{theorem}[{\cite{h12}}]
Let $\chi_G(q) = a_n q^n - a_{n-1}q^{n-1} + \cdots + (-1)^n a_0$ be the chromatic polynomial of a graph $G$.  Then, the sequence $a_0,\ldots,a_n$ is log-concave.
\end{theorem}

We now strengthen Conjecture~\ref{conj:dyckLor} to the class of stable polynomials, which are a multivariate version of real-rooted polynomials.  A polynomial $f \in \R[x_1,\ldots,x_k]$ is \emph{stable} if it has no roots in the product of $k$ open upper half-planes. We point to \cite{Wagner} for a survey on stable polynomials, as well as to the papers \cite{BB1,BB2} by Borcea and Br\"and\'en for more theory on stable polynomials.

We note that the class of Lorentzian polynomials agrees with the class of homogeneous stable polynomials for quadratic polynomials, but is larger for other degrees.  For example, (normalized) Schur polynomials are Lorentzian but not stable in general \cite[Example 9]{HMMStD}.

Unfortunately, checking stability is harder than checking the Lorentzian property.  In particular, one can check that a polynomial is stable by checking that an infinite number of certain univariate specializations are real-rooted \cite[Lemma 2.3]{Wagner}. Using SageMath \cite{Sage-Combinat}, we probed a random assortment of such univariate specializations to make the following conjecture.

\begin{conjecture}\label{conj:dyckstab}
Let \(d\) be a Dyck path. Then $X_{G(d)}$, restricted to any finite number of variables, is stable.
\end{conjecture}

\begin{example}
For the Dyck path $d=\sfn\sfn\sfe\sfn\sfe\sfe\sfn\sfe$, we have that $\lambda^{\gr}(d)=(3,1)$, $X_{G(d)} =
24 m_{1111} + 8m_{211} + 2m_{22} + m_{31}$, and $\Newton(X_{G(d)}(x_1,\ldots,x_k))=\mathcal{P}^{(k)}_{31}$. One can check that $X_{G(d)}$ is Lorentzian and see Figure~\ref{fig:ex newton polytope coeffs} for a diagram of its Newton polytope with coefficients exhibiting log-concavity in root directions.
\end{example}

We conclude this subsection with an example showing that incomparability graphs of (3+1)-free posets are not Lorentzian, and thus not stable. 

\begin{example}\label{ex:faillor}

Let \(G = C_4\) be the \(4\)-cycle, which is co-bipartite. Note that \(G\) is the incomparability graph of the (2+2)-poset, which is (3+1)-free. It has chromatic symmetric function  $X_{C_4} = 24 m_{1111} + 4m_{211} + 2m_{22}$.
The polynomial $f=X_{C_4}(x_1,\ldots,x_5)$ is M-convex but is not Lorentzian since the quadratic form associated to $\dfrac{\partial}{\partial x_1} \circ \dfrac{\partial}{\partial x_2} f$, which has matrix
\[
A= \begin{pmatrix}
0 & 8 & 8 & 8 & 8 \\
8 & 0 & 8 & 8 & 8 \\
8 & 8 & 8 & 24 & 24 \\
8 & 8 & 24 & 8 & 24 \\
8 & 8 & 24 & 24 & 8
\end{pmatrix},
\]
with characteristic polynomial $(x + 8)(x+16)^2(x^2 - 64x + 64)$, has two positive eigenvalues.
\end{example}

\begin{figure}
    \centering
    \includegraphics[scale=0.4]{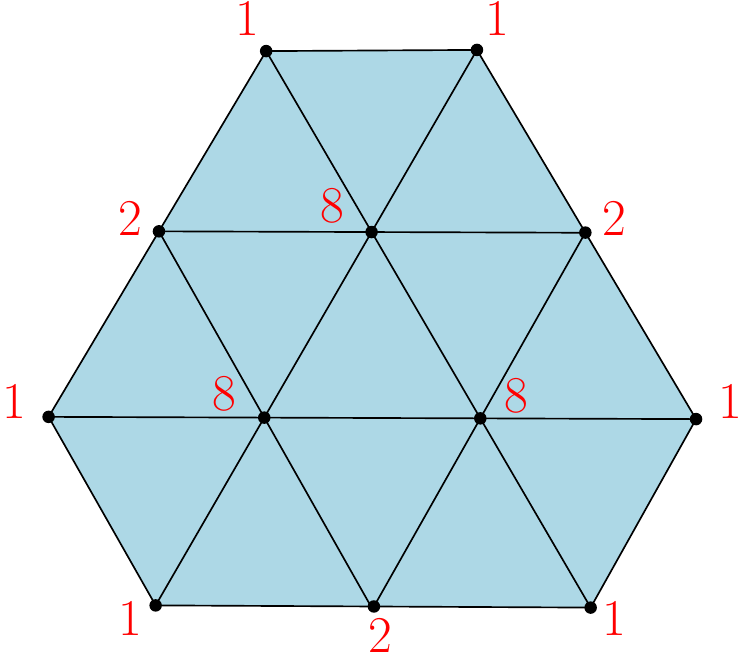}
    \caption{The Newton polytope $\mathcal{P}^{(3)}_{31}$  of $X_{G(d)}(x_1,x_2,x_3)$, for $d=\sfn \sfn\sfe \sfn\sfe\sfe\sfn\sfe$, with the coefficient of each lattice point in \textcolor{red}{red}.}
    \label{fig:ex newton polytope coeffs}
\end{figure}

\subsection{Lorentzian property for abelian Dyck paths}

In this section we verify Conjecture~\ref{conj:dyckLor} for abelian Dyck paths, i.e. paths whose indifference graphs $G(d)$ are co-bipartite.

\begin{theorem} \label{thm: Lorentzian abelian case}
Let $d$ be an abelian Dyck path. Then $X_{G(d)}$ is Lorentzian.
\end{theorem}

\begin{proof}
Let $d$ be an abelian path of size $n=n_1+n_2$ whose co-bipartite indifference graph  $G(d)$ has vertex set $\{1,\ldots,n_1\}\cup \{n_1+1,\ldots,n_1+n_2\}$ and is encoded by a Ferrers boards $B_{\mu} \subset [n_1]\times [n_2]$ of partitions $\mu=(\mu_1,\ldots,\mu_{\ell})$. By \eqref{eq: monomial expansion co-bipartite} we have that 
\begin{equation} \label{eq: monomial abelian case}
X_{G(d)}= \sum_{i} i! \cdot (n-2i)!\cdot r_i \cdot m_{2^i1^{n-2i}},
\end{equation}
where $r_i=r_i(B_{\mu})$ is the number of placements of $i$ non-attacking rooks in $B_{\mu}$. 

By Corollary~\ref{cor: M-convex for Dyck paths} we know that $X_{G(d)}$ is M-convex.
By Definition~\ref{def:lor}, showing that the symmetric polynomial \(X_{G(d)}(x_1, \ldots, x_k)\) is Lorentzian amounts to checking that for each partition ${\bf \alpha}=(\alpha_1,\ldots,\alpha_k)$ of $n-2$, the $k\times k$ matrix $H_{\bf \alpha}=( ({\bf \alpha} + e_r+e_s)!\cdot c_{{\bf \alpha}+e_r+e_s})_{r,s=1}^k$ has at most one positive eigenvalue, where \(c_{\bf \alpha}\) is the coefficient of \({\bf x}^{\bf \alpha}\) in \(X_{G(d)}\).

By \eqref{eq: monomial abelian case}, the support of $X_{G(d)}(x_1, \ldots, x_k)$ is in \(\{0,1,2\}^k \subset \mathbb{N}^k\).  Thus, for $k >  n$ variables, we only have to consider the matrices $H_{\alpha}$ of the partition 
\begin{equation} \label{possible comps}
{\bf \alpha}=(2^{i-1},1^{n-2i},0^{k-n+i+1}),
\end{equation}
for $i \geq 1$. The matrix $H_{\bf \alpha}$ has the form %\jm{say something about the latter?}
\[
H_{\bf \alpha} = 
\raisebox{-18pt}{\includegraphics{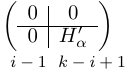}}, \quad H'_{\bf \alpha}= M_{n-2i-1,k-n+i}(a,b,c)  
\,\,\text{where}\,\, 
\begin{array}{c@{}c@{}l}
a&\,=&\, 2^{i+1}\cdot(i+1)!\cdot (n-2i-2)!\cdot  r_{i+1}\\
b&\,=&\,2^i \cdot i!\cdot (n-2i)!\cdot r_i\\
c&\,=&\,2^{i-1}\cdot (i-1)!\cdot (n-2i+2)!\cdot r_{i-1},
\end{array}
\]
and $M_{p,q}(a,b,c)$ is the block matrix in Figure~\ref{fig: block matrix}. 

The characteristic polynomial of $H'_{\bf \alpha}$ is given, via Proposition~\ref{prop: det block matrix}, by
\begin{multline} \label{eq: char poly}
\det(xI-H'_{\bf \alpha}) = (x+a)^{n-2i-1}\cdot (x-b+c)^{k-n+i} \cdot \left(x^2 - ((n-2i-1)a+b +(k-n+i)c)x  \right.\\
\left. -(n-2i)(k-n+i+1)b^2 + (n-2i-1)a(b+(k-n+i)c)\right).
\end{multline}

So  $X_{G(d)}$ is Lorentzian if and only if the polynomial in \eqref{eq: char poly} always has at most one positive root. This fact is implied by the following inequalities:
\begin{align}
b-c &\leq 0, \label{1st condition} \\
-(n-2i)(k-n+i+1)b^2 + (n-2i-1)a(b+(k-n+i)c) &\leq 0 \label{2nd condition},
\end{align}
where \eqref{1st condition} comes from a root of the polynomial and \eqref{2nd condition} follows from the quadratic formula. These two inequalities are verified in Propositions~\ref{prop:1st condition holds} and \ref{prop:2nd condition holds} below.
\end{proof}

The next result gives a formula for the characteristic polynomials of block matrices like $H'_{\alpha}$. For indeterminates $a,b,c$  and nonnegative integers $p,q\geq 0$ let $M_{p,q}(a,b,c)$ be the block matrix in Figure~\ref{fig: block matrix}.

\begin{figure}
    \centering
      \centering
  \begin{subfigure}[b]{0.27\textwidth}
    \centering
     \includegraphics{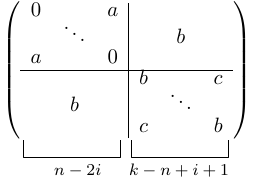}
    \caption{}
    \label{fig: block matrix}
\end{subfigure}
\begin{subfigure}[b]{0.7\textwidth}
  \centering
     \includegraphics[scale=0.96]{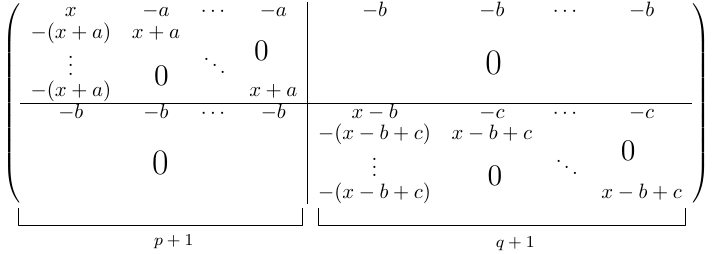}
  \caption{}
  \label{fig: block matrix row reduction}
\end{subfigure}
    \caption{(A) The block matrix $M_{p,q}(a,b,c)$ and (B) the block matrix $N_{p,q}(x;a,b,c)$ obtained from  $xI-M_{p,q}(a,b,c)$ by doing certain row operations.}
\end{figure}

\begin{proposition} \label{prop: det block matrix}
For indeterminates $a,b,c$ and nonnegative integers $p,q\geq 0$, the matrix $M_{p,q}(a,b,c)$ has characteristic polynomial
\[
\det(xI-M_{p,q}(a,b,c)) = (x+a)^p(x-b+c)^q(x^2-x(pa+b+qc)-(p+1)(q+1)b^2+pa(b+qc)).
\]
\end{proposition}

\begin{proof}
We subtract the first row of $xI-M_{p,q}(a,b,c)$ from rows $2$ to $p+1$ and we subtract row $p+2$ from rows $p+3$ to $p+q+2$ to obtain the matrix $N_{p,q}(x;a,b,c)$ in Figure~\ref{fig: block matrix row reduction}. The determinant remains unchanged. Next, we partition the matrix into the same blocks as in the figure and use the {\em Schur complement} (see \cite[\S 0.3]{SchCbook}) to calculate the determinant.  Hence
\begin{align}
\det(xI-M_{p,q}(a,b,c)) &= \det N_{p,q}(x;a,b,c) \notag \\
&= \det \left(\begin{array}{c|c} A & B \\ \hline  C & D \end{array}\right) = \det(A)\cdot \det(D - CA^{-1}B), \label{eq:Schur complement}
\end{align}
where 
\[
CA^{-1}B = \frac{(p+1)b^2}{x-pa}\begin{pmatrix}
1 &\cdots &1     \\
& &     \\
&0 &   \\
&  &  
\end{pmatrix}, \quad 
D- CA^{-1}B = \raisebox{-20pt}{\includegraphics{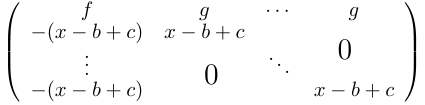}},
\]
for $f=x-b - (p+1)b^2/(x-pa)$ and $g=-c-(p+1)b^2/(x-pa)$. 
By doing a cofactor expansion, say on the first row of $A$ and $D-CA^{-1}B$, one readily obtains that 
\begin{align*}
\det(A) &= (x+a)^p\cdot (x-pa),\\
\det(D-CA^{-1}B) &= (x-b+c)^q(f +qg) \\
%\det(D-CA^{-1}B)
&=\frac{(x-b+c)^q}{x-pa} \left(x^2 - x(pa+b+qc) - (p+1)(q+1)b^2+pa(b+qc)\right).
\end{align*}
Using these two formulas in \eqref{eq:Schur complement} gives the desired result.
\end{proof}

The rest of this section is devoted to verifying \eqref{1st condition} and \eqref{2nd condition}. The next result shows \eqref{1st condition}, which is true for all co-bipartite graphs.

\begin{proposition}\label{prop:1st condition holds}
Let $G$ be a co-bipartite graph with vertex set $\{1,\ldots,n_1\}\cup \{n_1+1,\ldots,n_1+n_2\}$. Then \eqref{1st condition} holds for all \(i\); that is, for $i\geq 1$ we have 
\[
2\cdot i!\cdot (n_1+n_2-2i)!\cdot r_{i}(B)  \leq (i-1)!\cdot (n_1+n_2-2i+2)!\cdot r_{i-1}(B),
\]
where  $B \subset [n_1]\times [n_2]$ is the board associated to $G$.
\end{proposition}

\begin{proof}
For convenience, we substitute \(j = i - 1\).
There are $(j+1)\cdot r_{j+1}(B)$ placements of $j+1$ non-attacking rooks in $B$ with a distinguished rook. An overcount of this quantity is the number of pairs $(p,c)$, where $p$ is a placement of $j$ non-attacking rooks in $B$ and $c$ is a cell in $[n_1]\times [n_2]$ in a different row and column than the $j$ rooks. Thus we have 
\begin{equation}
\label{ieq: key b<=c}
(j+1)\cdot r_{j+1}(B) \leq (n_1-j)(n_2-j)\cdot r_j(B).
\end{equation}
Without loss of generality, assume $n_1\leq n_2$. Then $n_1+n_2-2j \geq 2(n_1-j)$. The desired inequality is trivially true if \(j+1 > n_1\) or \(j+1 > n_2\), since then \(r_{j+1}(B) = 0\).
So, we can assume \(j + 1 \leq n_1 \leq n_2\).
Thus, $n_1+n_2-2j-1 \geq n_2-j$.
Thus we have that
\begin{align*}
2(n_1-j)(n_2-j)\cdot r_j(B) &\leq
(n_1+n_2-2j)(n_1+n_2-2j-1) \cdot r_j(B).
\end{align*}
Combining this inequality with \eqref{ieq: key b<=c}, we obtain
\[
2\cdot (j+1)\cdot r_{j+1}(B) \leq (n_1+n_2-2j)(n_1+n_2-2j-1)\cdot r_{j}(B),
\]
which is equivalent to the desired result. 
\end{proof}

We now verify \eqref{2nd condition}, which is true for Ferrers boards but not necessarily all boards (see Examples~\ref{ex:faillor} and \ref{ex: diagonal rooks}). We need the following lemma that follows from a result of Haglund--Ono--Wagner \cite{HOW} about the {\em ultra log-concavity} of {\em hit numbers} of Ferrers boards. Note that ultra log-concavity of rook numbers, which holds for all boards (see \cite{JH,HOW}), is not sufficient to ensure \eqref{2nd condition} (see Example~\ref{ex: diagonal rooks}).

\begin{lemma}\label{lem:UULC rooks}
Suppose \(\mu = (\mu_1, \ldots, \mu_\ell)\) is a partition. Then for $i\geq 1$, the rook numbers \(r_i \coloneq r_i(B_{\mu})\) satisfy 
\begin{equation} \label{eq: UULC rooks}
r_i^2 \geq \left(1+\frac{1}{i}\right)\left(1+\frac{1}{\ell-i}\right)\left(1 + \frac{1}{\mu_1 - i}\right) r_{i-1}r_{i+1}. 
\end{equation}

\end{lemma}

\begin{proof}
The {\em hit} polynomial of a Ferrers board $B_{\mu} \subset [N] \times [N]$ is given by \[
T(x;\mu) \coloneq \sum_{i=0}^{N} (N-i)!\cdot r_i(B_{\mu})\cdot (x-1)^{i},
\]
where \(N\) must be big enough to contain \(\mu\). Assume without loss of generality that \(\mu_1 \geq \ell\), and take \(N = \mu_1\).

Haglund--Ono--Wagner \cite[Theorem 1]{HOW} showed that $T(x;\mu)$ is real-rooted, so this is also true for \(T(x+1;\mu)\). Furthermore, the degree of \(T(x+1;\mu)\) is at most \(\ell\), since no more than \(\ell\) rooks can be placed on \(\mu\).
Newton's inequality (see, e.g., \cite[p. 52]{HLP}) tells us that the coefficients of this polynomial are {\em ultra log-concave}. This means that the sequence
\[
\frac{(\mu_1-i)!}{\binom{\ell}{i}}r_i
\]
is log-concave. That is,
\[
i(\ell-i)\cdot (\mu_1-i)!^2 \cdot r_i^2\geq (i+1)(\ell-i+1)(\mu_1-i+1)! (\mu_1-i-1)!\cdot r_{i-1} \cdot r_{i+1}, 
\]
which is equivalent to the desired result.
\end{proof}

We are now ready to verify \eqref{2nd condition}.

\begin{proposition} \label{prop:2nd condition holds}
Equation~\eqref{2nd condition} holds.
\end{proposition}

\begin{proof}
\(G(d)\) is encoded by a partition \(\mu = (\mu_1, \ldots, \mu_\ell)\) inside \([n_1]\times[n_2]\), with \(\text{deg}(X_{G(d)}) = n = n_1 + n_2\).
Assume without loss of generality that \(\mu_1 \geq \ell\).
By Lemma~\ref{lem:UULC rooks} the following inequality is true
\begin{equation}
    r_i^2 \geq \left(1+\frac{1}{i}\right)\left(1+\frac{1}{\ell-i}\right)\left(1+\frac{1}{\mu_1-i}\right) r_{i-1}r_{i+1}.
\end{equation}

Using \(i \leq \ell \leq \mu_1\) and \(\ell + \mu_1 \leq n\) gives
\begin{equation} \label{eq: UULC 2}
 r_i^2 \geq
    \left(1+\frac{1}{i}\right)\left(1+\frac{2}{n-2i}\right)\left(1+\frac{1}{n-2i}\right) r_{i-1}r_{i+1},
\end{equation}
which is equivalent to
\begin{equation} \label{eq: UULC 3}
    (n-2i)b^2 \geq (n-2i-1)ac.
\end{equation}

Multiplying both sides of this inequality by $k-n+i+1\geq 0$ and using 
\eqref{1st condition} gives the desired result.
\end{proof}

\begin{example} \label{ex: diagonal rooks}
Continuing with Example~\ref{ex:faillor}, the $4$-cycle $C_4$ is a co-bipartite graph associated to the diagonal board $B=\{(1,1),(2,2)\} \subset [2]\times [2]$. For this board we have that $r_0=1$, $r_1=2$, and $r_2=1$, so for $i=1$ we have  
\[
4=r_1^2 <  \left(1+\frac{1}{i}\right)\left(1+\frac{2}{n-2i}\right)\left(1+\frac{1}{n-2i}\right) r_{i-1}r_{i+1} = 2\cdot 2 \cdot \frac32 \cdot 1 \cdot 1.
\]
Thus \eqref{eq: UULC 2} does not hold. And for $k>4$ variables, neither do \eqref{eq: UULC 3} or \eqref{2nd condition}. 
\end{example}

%%%%%%%%%%%%%%%%%%%%%%%%%%%%%%%%%%%%%%%%%%%%%%%%%%%%%%%%%%%%%%%%%%%%%%%%%%%
\section{Further examples and conjectures} 
\label{sec: conjecture}
%%%%%%%%%%%%%%%%%%%%%%%%%%%%%%%%%%%%%%%%%%%%%%%%%%%%%%%%%%%%%%%%%%%%%%%%%%%

\subsection{Relation with the $\zeta$ map from diagonal harmonics} \label{sec: listing to Dyck paths}

We have two Dyck paths associated to a (3+1)- and (2+2)-free poset (i.e. a unit interval order) $P$ of size $n$: $P$ corresponds to an incomparability graph $G(d)$ of a Dyck path $d$ and  to a lex-maximal part listing $v_{a_1}\cdots v_{a_n}$ of an area sequence ${\bf a}= (a_1,\ldots,a_n)$ of a Dyck path $d'$ by Theorem~\ref{thm: char unit interval orders} and Remark~\ref{def: varphi area sequence}. Using FindStat \cite[\href{http://www.findstat.org/MapsDatabase/Mp00032oMp00028}{link}]{FindStat}, it appears that these Dyck paths are related by Haglund's well-known $\zeta$ map from diagonal harmonics (e.g. see \cite[Theorem 3.15]{HagqtCat}). See \cite[Remark 6.6]{FennSommers} for a similar statement in terms of ad-nilpotent ideals.

\begin{conjecture}\!\!\!\footnote{This conjecture has  been proved independently by G\'elinas--Segovia--Thomas \cite{gelinas2022proof} and by Fang \cite{fang2022bijective}.}  \label{conj: listing to Dyck path}
Let $P$ be a unit interval order corresponding to an incomparability graph $G(d)$ and a lex-maximal part listing encoded by a tuple  ${\bf a}=(a_1,\ldots,a_n)$. If $d'$ is the Dyck path with area sequence ${\bf a}$, then 
\[
d = \zeta(d').
\]
\end{conjecture}

\begin{example}
The unit interval order $P$ associated to the Dyck path $d=\sfn\sfn\sfe\sfe\sfn\sfn\sfe\sfe$ in Figure~\ref{fig:ex_path_graph} corresponds to the lex-maximal listing $v_{0}v_0v_1v_1v_1v_0$. The associated tuple ${\bf a} = (0,0,1,1,0)$ is the area sequence of the Dyck path $d'=\sfn\sfe\sfn\sfn\sfe\sfn\sfe\sfe\sfn\sfe$. One can check that $d=\zeta(d')$, as illustrated in Figure~\ref{fig: example zeta map}. For a larger example, see Figure~\ref{fig:big example Dyck path correspondence}.
\end{example}

\begin{figure}
    \centering
    \includegraphics{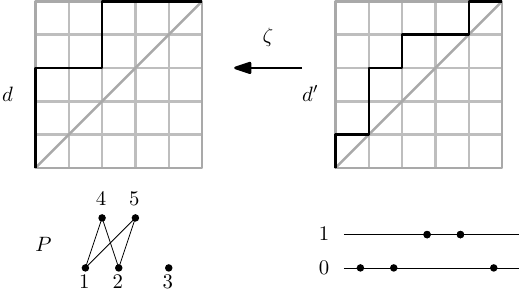}
    \caption{There are two Dyck paths associated to a unit interval order $P$: one is the path $d$ whose indifference graph is the incomparability graph of $P$, and the other one is the path $d'$ associated to the lex-maximal listing tuple ${\bf a}$. Conjecture~\ref{conj: listing to Dyck path} states that $d=\zeta(d')$ as illustrated in this example.}
    \label{fig: example zeta map}
\end{figure}

\subsection{Chromatic symmetric functions with reflexive Newton polytopes}

An important property in the {\em Ehrhart theory} of lattice polytopes, i.e. polytopes with integral vertices, is that of a  polytope being reflexive \cite{Br}. A lattice polytope $Q$ with ${\bf 0}$ in its interior is \emph{reflexive} if its polar (dual) $Q^*$ is a lattice polytope (see, e.g., \cite[Sec. 4.4]{CCD}). In \cite[Theorem 34]{bayer2020lattice} the authors characterized when a permutahedron $\mathcal{P}^{(k)}_{\lambda}$ is reflexive.  In Section~\ref{sec: case dyck paths} we showed that for a Dyck path $d$, the Newton polytope of $X_{G(d)}(x_1,\ldots,x_k)$ is the permutahedron $\mathcal{P}^{(k)}_{\lambda^{\gr}(d)}$. It would be interesting to use their characterization of reflexive permutahedra to find all Dyck paths $d$ for which the Newton polytope of $X_{G(d)}(x_1,\ldots,x_k)$ is reflexive.

\subsection{Unimodality of colorings}

Although we have not been able to show Conjecture~\ref{conj:dyckLor} for arbitrary indifference graphs, which by Proposition~\ref{thm: log concavity of LP} would imply log-concavity of the coefficients, the following weaker result follows from Gasharov's $s$-positivity of $X_{G(d)}$ \cite{G2}.

\begin{proposition}
For a Dyck path $d$ with $X_{G(d)} = \sum_{\lambda} c^d_{\lambda} \cdot m_{\lambda}$, if $\mu \succeq \nu$ then $c^d_{\mu} \leq c^d_{\nu}$. 
\end{proposition}

\begin{proof}
Gasharov proved that $X_{G(P)}({\bf x})$ is $s$-positive \cite{G2}, thus
\[
X_{G(d)} = \sum_{\lambda} f^{d}_{\lambda} s_{\lambda},
\]
where $f_{\lambda}^d \in \mathbb{N}$.  Every Schur function has a monomial expansion of the form $s_{\lambda} = \sum_{\mu} K_{\lambda\mu} m_{\mu}$. In this expansion, if $\mu \succeq \nu$, then we have the inequality $K_{\lambda\mu} \leq K_{\lambda\nu}$ of Kostka numbers (see  \cite[Ex. 9 (b) \SS 1.7]{Mac} or \cite{Whi}). Thus if   $\mu \succeq \nu$ then
\[
c^d_{\mu} = \sum_{\lambda} f^d_{\lambda} K_{\lambda\mu} \leq \sum_{\lambda} f^d_{\lambda} K_{\lambda\nu} = c^d_{\nu},
\]
as desired.
 \end{proof}

\subsection{SNP property of chromatic symmetric functions}\label{snpsection}

We show in Theorem~\ref{thm:p_lambda 3+1 free} that (3+1)-free incomparability graphs have permutahedral support, so they are all M-convex and have the SNP property.

\begin{remark}
This result does not hold for analogous graphs which are not incomparability graphs of posets. If \(G\) is an incomparability graph of a poset \(P\), it is claw-free if and only if \(P\) is (3+1)-free. But there are claw-free graphs for which the chromatic symmetric function does not even have M-convex support (see Example \ref{example:triforce}).
\end{remark}

\begin{example}\label{example:triforce}
Let $G$ be the claw-free graph with six vertices in Figure~\ref{fig:counter1}. Note that (when expanded in \(6\) variables) $X_G = 162s_{1^6} +72 s_{21^4} + 12 s_{2^21^2} + 6 s_{2^3} + 6s_{31^3}$ is SNP; however, it is not M-convex since $(1, 1, 1, 3, 0, 0)$, $(0, 0, 2, 2, 2, 0)$ are both in the support, but
\[
(0,0,2,2,2,0) + e_4 - e_i
\]
is not for any \(i\) in \(\{1,2,3,5\}\).
\end{example}

Conjecture \ref{conj:mon} says that the chromatic symmetric function of any $s$-positive graph should be SNP. Our result that (3+1)-free incomparability graphs have chromatic symmetric functions with M-convex support is a partial confirmation of the conjecture.

However, in order to test the conjecture for other graphs one needs to look at graphs with size $n\geq 12$. The next minimal example shows that there are $s$-positive symmetric functions that are not SNP, but they do not occur for small $n$.\footnote{There are no $3$-antichains of partitions of $n<12$ in dominance order with one partition being a convex combination of the other two.} This makes it hard to find a counterexample for Conjecture~\ref{conj:mon}.

\begin{example}
The function
\[
f = s_{6222} + s_{444}
\]
is not SNP (when expanded in at least 4 variables). The vector \((5,3,3,1,0,\ldots)\) is a convex combination \(\frac{1}{2}(6,2,2,2) + \frac{1}{2}(4,4,4,0)\), but the partition \((5,3,3,1)\) is not dominated by either \((6,2,2,2)\) or \((4,4,4)\), so it is not in the support of \(f\).
\end{example}

\begin{remark}
Furthermore, there are $s$-positive incomparability graphs that contain claws, which are not covered by our Theorem \ref{thm:p_lambda 3+1 free}: see Example \ref{example:6tree}. These can fail to be M-convex (as in the example), and it seems plausible that an $s$-positive incomparability graph with 12 vertices that contains claws could fail to be SNP. We were unable to complete a search over the space of incomparability graphs with 12 vertices due to computational constraints.
\end{remark}

\begin{example}\label{example:6tree}
Let $G$ be the tree with six vertices in Figure~\ref{fig:counter2}, which is an incomparability graph for the poset in Figure~\ref{fig:counter_poset}. Then $X_G = 32s_{1^6} + 40s_{21^4} + 18s_{2^21^2} + 8 s_{2^3} + 16s_{31^3} + 6s_{321} + 2s_{3^2} + 2 s_{41^2}$. This is not M-convex since \((0, 0, 0, 3, 0, 3)\) and \((0, 0, 0, 4, 1, 1)\) are both in the support, but
\[
(0,0,0,3,0,3) + e_4 - e_6
\]
is not. 
%Also note that $G$ can be colored with weights $(6,2,2,2)$ and $(4,4,4)$, but not with weight $(5,3,3,1)$.
\end{example}

\begin{figure}
  \begin{subfigure}[b]{0.3\textwidth}
    \centering
     \includegraphics[scale=0.8]{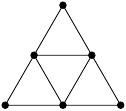}
    \caption{}
\label{fig:counter1}
\end{subfigure}
\begin{subfigure}[b]{0.3\textwidth}
  \centering
  \includegraphics[scale=0.8]{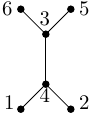}
  \caption{}
  \label{fig:counter2}
\end{subfigure}
\begin{subfigure}[b]{0.3\textwidth}
  \centering
  \includegraphics[scale=0.8]{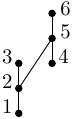}
  \caption{}
  \label{fig:counter_poset}
\end{subfigure}
 \caption{Examples of graphs with $s$-positive chromatic symmetric functions that are not M-convex. The example in (A) is claw-free and chordal but is not an incomparability graph. The example in (B) 
    is a chordal incomparability graph of the poset in (C), but the graph contains a claw.}

    \centering
    
     \quad

    \label{fig:counter}
\end{figure}

\subsection{Complexity of $X_{G(P)}$ and $X_{G(d)}$} \label{complexitysubsection}

The study of the complexity of chromatic symmetric functions of general graphs and claw-free graphs was started by Adve--Robichaux--Yong \cite{ARY}. We give some preliminary results on these questions for graphs $G(d)$ and the more general $G(P)$.  

Given a (3+1)-free poset $P$ (resp. a co-bipartite graph $G$ or a Dyck path $d$) and its chromatic symmetric function $X_{G(P)}=\sum_{\alpha} c^P_{\alpha} {\bf x}^{\alpha}$  (resp. $X_G = \sum_{\alpha} c^G_{\alpha} {\bf x}^{\alpha}$ or $X_{G(d)} = \sum_{\alpha} c^d_{\alpha} {\bf x}^{\alpha}$), it is of interest to study the {\em nonvanishing decision problem}: the complexity of deciding whether $c^P_{\alpha} \neq 0$ (resp. $c^G_{\alpha}\neq 0$ or $c^d_{\alpha}\neq 0$) and the complexity of computing $c^P_{\alpha}$ (resp. $c^G_{\alpha}$ or $c^d_{\alpha}$), both measured in the input size of $P$ (resp. $G$ and $d$).  For the sake of specificity, we assume a Dyck path $d$ of length $n$  is given as a length $2n$ string of $\sfe$ and $\sfn$ steps. A poset $P$ is specified by a list of its cover relations, and a co-bipartite graph is specified by a list of its edges.

\begin{proposition}\label{prop:non vanishing dyck paths}
Let $d$ be a Dyck path of length $n$. Given a weight \(\alpha \in \mathbb{N}^n\), deciding whether \(c^d_{\alpha}\) is nonzero is in \(\mathsf{P}\) (takes time polynomial in \(n\)).
\end{proposition}

\begin{proof}
By Theorem \ref{thm:p_lambda}, the support of \(X_{G(d)}(x_1,\ldots,x_k)\) is the permutahedron \(\mathcal{P}^{(k)}_{\lambda^{\gr}(d)}\). The greedy algorithm to determine \(\lambda^{\gr}(d)\) from \(d\) takes time polynomial in \(n\):
for each number \(i\) in $[n]$, consider vertex \(i\). For each other vertex \(j\), check if \(j\) is connected and add it to the list of neighbors of \(i\) if so. Consider each color \(x\) in order, and if \(x\) is not a color of a neighbor of \(i\), color the vertex \(i\) the color \(x\) and move on to the next vertex.
(It suffices to consider each pair of vertices only once.) Once \(\lambda^{\gr}(d)\) is determined, determining membership of \(\alpha\) in the permutahedron takes polynomial time as well by Rado's theorem \cite{Rado}.
\end{proof}

\begin{proposition} \label{prop:nonvanishing 3+1 free}
Let $P$ be a (3+1)-free poset on $n$ vertices. Given a weight \(\alpha \in \mathbb{N}^n\), deciding whether \(c^P_{\alpha}\) is nonzero is in \(\mathsf{P}\) (takes time polynomial in \(n\)).
\end{proposition}

\begin{proof}
Recall that $P$ is specified as a list of cover relations. Following the decomposition in \cite{GPMR,MGP}, we can convert $P$ into a part listing $L$ in polynomial time (where the bicolored graphs in $L$ are encoded as adjacency matrices). 

Following our proof of Theorem~\ref{thm:p_lambda 3+1 free}, we find the dominating weight $\lambda^{\gr}(P)$ by finding for each bicolored graph $H$ the maximum $U_k$ (or $D_k$) appearing in its convex decomposition. Such $k$ is the size of the maximum matching in $H$, which we can find in polynomial time (see \cite[Section 16.4]{ShrijverA}). The result then follows by the same argument as in the proof of Proposition~\ref{prop:non vanishing dyck paths}.
\end{proof}

Next we use Lemma~\ref{lem:StanleyStembridgeCobipartite} to determine the complexity of computing the coefficients of $X_G$ in the monomial basis for co-bipartite graphs. 

\begin{proposition} \label{prop:complexity-cobipartite}
If \(G\) is a co-bipartite graph, then determining the coefficients \(c^G_{\alpha}\) is \(\#\mathsf{P}\)-complete.
\end{proposition}

\begin{proof}
Computing the permanent of a $0$-$1$ matrix $A$ of size $n\times n$ is $\#\mathsf{P}$-complete \cite{Valiant}. If $B \subset [n]\times [n]$ is the support of the matrix $A$, then $\perm(A)=r_n(B)$. Given the board $B$, let $G$ be the co-bipartite graph with two cliques on vertices $\{1,\ldots,n\} \cup \{n+1,\ldots, 2n\}$ and edges $(i,n+j)$ for each $(i,j)$ not in $B$. Then by \eqref{eq:rel_rooks} we have that $c^G_{2^n}= n! \cdot r_n(B)=n! \cdot \perm(A)$. Hence, determining the coefficients $c^G_{\alpha}$ of $X_G$ is \(\#\mathsf{P}\)-complete as desired.
\end{proof}

Since co-bipartite graphs are incomparability graphs of (3+1)-free posets, we immediately obtain the following.

\begin{corollary}\label{cor:3+1 coeffs}
If $P$ is a (3+1)-free poset, then determining the coefficients \(c^P_{\alpha}\) is \(\#\mathsf{P}\)-complete.
\end{corollary}

\begin{proof}
The result follows from Proposition~\ref{prop:complexity-cobipartite} and the fact that co-bipartite graphs are incomparability graphs of 3-free posets.
\end{proof}

\begin{remark}
Given a Dyck path $d$, it would be interesting to see whether or not determining the coefficients \(c^d_{\alpha}\) of $X_{G(d)}$  is \(\#\mathsf{P}\)-complete. More concretely, is determining the leading coefficient \(c^d_{\lambda^{\gr}(d)}\) for the greedy coloring \(\#\mathsf{P}\)-complete?
\end{remark}

\begin{remark}
In contrast, one can compute \(c^d_{2^k1^{n-2k}}\) in polynomial time for abelian Dyck paths $d$ (i.e. Dyck paths with indifference graphs $G(d)$ that are also co-bipartite), which are encoded by Ferrers boards $B_{\mu}$ of partitions $\mu=(\mu_1,\ldots,\mu_{\ell})$. Then by classical rook theory \cite{kaplansky1946problem}, $\sum_k r_k(B_{\mu}) (x)_{r-k} = \prod_i (x+\mu_{\ell-i+1}-i-1)$, where $(x)_m = x(x-1)\cdots (x-m+1)$. The coefficients $r_k(B_{\mu})$ can be extracted using the {\em Stirling numbers} of the second kind $S(m,k)$, since $x^m=\sum_{k=0}^m S(m,k) (x)_k$. The numbers $S(m,k)$ can in turn be computed efficiently, say by using their recurrence (e.g. see \cite[Eq. 1.93, 1.96]{EC1}). 
\end{remark}

\begin{remark}
We know of two recent algorithms to compute $X_{G(d)}$, and it would be interesting to analyze their complexity.
\begin{itemize}
\item  Carlsson and Mellit \cite[Section 4]{CM}
defined chromatic symmetric functions of partial Dyck paths and
defined a {\em Dyck path algebra} generated by operators $D_{\sfn},
D_{\sfe}$ that act on these symmetric functions by adding north
steps $\sfn$ and east steps $\sfe$ to the Dyck path. These operators use {\em
  plethystic operations} (e.g. see \cite[Chapter 1]{HagqtCat}). If the Dyck path $d$ has steps  $\epsilon_1
\cdots \epsilon_{2n}$, then \cite[Theorem 4.4]{CM} implies that
\[
X_{G(d)} \,=\, D_{\epsilon_1}\cdots D_{\epsilon_{2n}} (1).
\]
\item Abreu and Nigro \cite[Algorithm 2.8]{AN} gave a recursive algorithm, based on the modular relation, to compute $X_{G(d)}$. 
\end{itemize}
\end{remark}

\subsection*{Acknowledgements}
This project started during a stay of the second named author at Institut Mittag-Leffler in
Djursholm, Sweden in the course of the program on Algebraic and Enumerative Combinatorics in Spring 2020. We thankfully acknowledge the support of the Swedish Research Council under grant no. 2016-06596, and thank Institut Mittag-Leffler for its hospitality. We thank Margaret Bayer, Petter Br\"and\'en, Tim Chow, Laura Colmenarejo, Benedek Dombos, Félix G\'elinas, Mathieu Guay-Paquet, Álvaro Gutiérrez, Jim Haglund, Chris Hanusa, June Huh, 
Ricky Liu, Khanh Nguyen Duc, Greta Panova, Adrien Segovia, Mark Skandera, Eric Sommers, Hugh Thomas, Cynthia Vinzant, and Andrew Tymothy Wilson for helpful comments. We also would like to thank the anonymous referee for edits and suggestions that improved the manuscript. The first named author was partially supported by Max Planck Institute for Mathematics (MPIM), the Hausdorff Research Institute for Mathematics (HIM), the Deutsche Forschungsgemeinschaft (DFG) under Germany's Excellence Strategy - GZ 2047/1 Projekt-ID 390685813, and a Simons Foundation Travel Support for Mathematicians Award MPS-TSM-00007970.  The second and third named authors were partially supported by NSF grants DMS-1855536 and DMS-2154019.

\appendix
\section{Counterexample to Conjecture~\ref{conj:dyckLor}} \label{appendix:counterex}
%\renewcommand{\thesubsection}{\Alph{subsection}}
%\numberwithin{example}{subsection}
%\numberwithin{remark}{subsection}

After the paper was published, Ricky Liu and Cynthia Vinzant (private communication) found the following counterexample to Conjecture~\ref{conj:dyckLor}.

\begin{example}
 For $n=8$ and the Dyck path $d=\sfn\sfn\sfe\sfn\sfn\sfe\sfe\sfn\sfe\sfn\sfn\sfe\sfe\sfn\sfe\sfe$ in Figure~\ref{fig:counterexLorzCSF}, we have that 
 \begin{multline*}
X_{G(d)} = 40320m_{1^8} + 13680m_{21^6}  + 4656m_{2^21^4} + 1620m_{2^31^2} +  600m_{2^4} + 1680 m_{31^5} + 552 m_{321^3} + 196m_{32^21} \\+ 48m_{3^21^2}  + 22m_{3^22} + 72m_{41^4} + 22m_{421^2} + 10m_{42^2}.
 \end{multline*}
 When expanded in $k=4$ variables and 
for the partition $\alpha=(1,1,2)$, the associated quadratic form has matrix
%associated matrix is 
\[
H_{\alpha} = 
\begin{pmatrix}
   0  &    0  &  960 &  1056\\
   0 &     0 &  1584 &  1728\\
 960 &  1584 &  1584 &  4704\\
1056 &  1728 &  4704 & 4704
\end{pmatrix},
\]
which has two positive eigenvalues. The same partition $\alpha$ also works to show that the normalization of $X_{G(d)}(x_1,x_2,x_3,x_4)$ is not  Lorentzian.
\end{example}

\begin{figure}
\centering
    \includegraphics[]{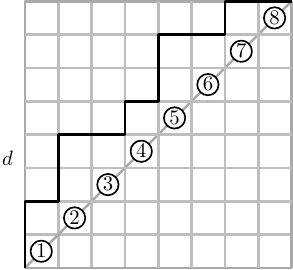}
    \caption{Counterexample to Conjecture~\ref{conj:dyckLor} for $n=8$ and $k=4$.}
    \label{fig:counterexLorzCSF}
\end{figure}

\bibliography{references}

% \bib, bibdiv, biblist are defined by the amsrefs package.
\begin{bibdiv}
\begin{biblist}

\bib{ABS}{article}{
      author={Aguiar, M.},
      author={Bergeron, N.},
      author={Sottile, F.},
       title={Combinatorial {H}opf algebras and generalized {D}ehn--{S}ommerville relations},
        date={2006},
     journal={Compos. Math.},
      volume={142},
      number={1},
       pages={1\ndash 30},
}

\bib{ALGVIII}{article}{
      author={Anari, N.},
      author={Liu, K.},
      author={Gharan, S.~O.},
      author={Vinzant, C.},
       title={Log-{C}oncave {P}olynomials {I}{I}{I}: {M}ason's {U}ltra-{L}og-{C}oncavity {C}onjecture for {I}ndependent {S}ets of {M}atroids},
        date={2024},
     journal={Proc. Amer. Math. Soc.},
      volume={152},
      number={5},
       pages={1969\ndash 1981},
}

\bib{AN}{article}{
      author={Abreu, A.},
      author={Nigro, A.},
       title={Chromatic symmetric functions from the modular law},
        date={2021},
        ISSN={0097-3165},
     journal={J. Combin. Theory Ser. A},
      volume={180},
       pages={Paper No. 105407, 30},
         url={https://doi.org/10.1016/j.jcta.2021.105407},
}

\bib{AP}{article}{
      author={Alexandersson, P.},
      author={Panova, G.},
       title={L{LT} polynomials, chromatic quasisymmetric functions and graphs with cycles},
        date={2018},
     journal={Discrete Math.},
      volume={341},
      number={12},
       pages={3453\ndash 3482},
}

\bib{ARY}{article}{
      author={Adve, A.},
      author={Robichaux, C.},
      author={Yong, A.},
       title={Computational complexity, {N}ewton polytopes, and {S}chubert polynomials},
        date={2020},
     journal={S\'{e}m. Lothar. Combin.},
      volume={82B},
       pages={Art. 52, 12},
}

\bib{ARY2}{article}{
      author={Adve, A.},
      author={Robichaux, C.},
      author={Yong, A.},
       title={An efficient algorithm for deciding vanishing of {S}chubert polynomial coefficients},
        date={2021},
     journal={Adv. Math.},
      volume={383},
       pages={Paper No. 107669, 38},
}

\bib{AS}{article}{
      author={Alexandersson, P.},
      author={Sulzgruber, R.},
       title={A combinatorial expansion of vertical-strip {LLT} polynomials in the basis of elementary symmetric functions},
        date={2022},
     journal={Adv. Math.},
      volume={400},
       pages={Paper No. 108256, 58},
}

\bib{BB1}{article}{
      author={Borcea, J.},
      author={Br\"{a}nd\'{e}n, P.},
       title={The {L}ee-{Y}ang and {P}\'{o}lya-{S}chur programs. {I}. {L}inear operators preserving stability},
        date={2009},
     journal={Invent. Math.},
      volume={177},
      number={3},
       pages={541\ndash 569},
}

\bib{BB2}{article}{
      author={Borcea, J.},
      author={Br\"{a}nd\'{e}n, P.},
       title={The {L}ee-{Y}ang and {P}\'{o}lya-{S}chur programs. {II}. {T}heory of stable polynomials and applications},
        date={2009},
     journal={Comm. Pure Appl. Math.},
      volume={62},
      number={12},
       pages={1595\ndash 1631},
}

\bib{BrCh}{article}{
      author={Brosnan, P.},
      author={Chow, T.~Y.},
       title={Unit interval orders and the dot action on the cohomology of regular semisimple {H}essenberg varieties},
        date={2018},
     journal={Adv. Math.},
      volume={329},
       pages={955\ndash 1001},
}

\bib{bayer2020lattice}{article}{
      author={Bayer, M.},
      author={Goeckner, B.},
      author={Hong, S.~J.},
      author={McAllister, T.},
      author={Olsen, M.},
      author={Pinckney, C.},
      author={Vega, J.},
      author={Yip, M.},
       title={Lattice polytopes from {S}chur and symmetric {G}rothendieck polynomials},
        date={2020},
     journal={Electronic Journal of Combinatorics},
      volume={27},
      number={P00},
}

\bib{bh20}{article}{
      author={Br\"{a}nd\'{e}n, P.},
      author={Huh, J.},
       title={Lorentzian polynomials},
        date={2020},
        ISSN={0003-486X},
     journal={Ann. of Math. (2)},
      volume={192},
      number={3},
       pages={821\ndash 891},
         url={https://doi-org.libproxy.uoregon.edu/10.4007/annals.2020.192.3.4},
}

\bib{b12}{article}{
      author={Birkhoff, G.~D.},
       title={A determinant formula for the number of ways of coloring a map},
        date={1912/13},
        ISSN={0003-486X},
     journal={Ann. of Math. (2)},
      volume={14},
      number={1-4},
       pages={42\ndash 46},
         url={https://doi.org/10.2307/1967597},
}

\bib{CCD}{book}{
      author={Beck, M.},
      author={Robins, S.},
       title={Computing the continuous discretely},
     edition={Second},
      series={Undergraduate Texts in Mathematics},
   publisher={Springer, New York},
        date={2015},
}

\bib{Br}{incollection}{
      author={Braun, B.},
       title={Unimodality problems in {E}hrhart theory},
        date={2016},
   booktitle={Recent trends in combinatorics},
      series={IMA Vol. Math. Appl.},
      volume={159},
   publisher={Springer, [Cham]},
       pages={687\ndash 711},
}

\bib{CC-RMM}{article}{
      author={Castillo, F.},
      author={Cid-Ruiz, Y.},
      author={Mohammadi, F.},
      author={Monta\~{n}o, J.},
       title={Double {S}chubert polynomials do have saturated {N}ewton polytopes},
        date={2023},
        ISSN={2050-5094},
     journal={Forum Math. Sigma},
      volume={11},
       pages={Paper No. e100},
         url={https://doi.org/10.1017/fms.2023.101},
}

\bib{TC}{misc}{
      author={Chow, T.},
       title={Note on the {S}chur-expansion of {$X_G$} for indifference graphs {$G$}},
        date={2015},
        note={\href{http://timothychow.net/firstfit.pdf}{link}},
}

\bib{CM}{article}{
      author={Carlsson, E.},
      author={Mellit, A.},
       title={A proof of the shuffle conjecture},
        date={2018},
     journal={J. Amer. Math. Soc.},
      volume={31},
      number={3},
       pages={661\ndash 697},
         url={https://doi.org/10.1090/jams/893},
}

\bib{CMP}{article}{
      author={Colmenarejo, L.},
      author={Morales, A.~H.},
      author={Panova, G.},
       title={Chromatic symmetric functions of {D}yck paths and {$q$}-rook theory},
        date={2023},
     journal={European J. Combin.},
      volume={107},
       pages={Paper No. 103595, 36},
}

\bib{CSSY}{article}{
      author={Chandler, A.},
      author={Sazdanovic, R.},
      author={Stella, S.},
      author={Yip, M.},
       title={On the strength of chromatic symmetric homology for graphs},
        date={2023},
        ISSN={0196-8858,1090-2074},
     journal={Adv. in Appl. Math.},
      volume={150},
       pages={Paper No. 102559, 34},
         url={https://doi.org/10.1016/j.aam.2023.102559},
}

\bib{tracesbook}{book}{
      editor={Diekert, V.},
      editor={Rozenberg, G.},
       title={The book of traces},
   publisher={World Scientific Publishing Co., Inc., River Edge, NJ},
        date={1995},
        ISBN={981-02-2058-8},
         url={https://doi.org/10.1142/9789814261456},
}

\bib{fang2022bijective}{article}{
      author={Fang, W.},
       title={Bijective proof of a conjecture on unit interval posets},
        date={2024},
     journal={Discrete Math. Theor. Comput. Sci.},
      volume={26},
      number={2},
       pages={Paper No. 2, 8},
}

\bib{FHM}{article}{
      author={Foley, A.~M.},
      author={Ho\`ang, C.~T.},
      author={Merkel, O.~D.},
       title={Classes of graphs with {$e$}-positive chromatic symmetric function},
        date={2019},
     journal={Electron. J. Combin.},
      volume={26},
      number={3},
       pages={Paper No. 3.51, 19},
         url={https://doi.org/10.37236/8211},
}

\bib{FMStD}{article}{
      author={Fink, A.},
      author={M{\'e}sz{\'a}ros, K.},
      author={St.~Dizier, A.},
       title={Schubert polynomials as integer point transforms of generalized permutahedra},
        date={2018},
     journal={Advances in Mathematics},
      volume={332},
       pages={465\ndash 475},
}

\bib{FennSommers}{article}{
      author={Fenn, M.},
      author={Sommers, E.},
       title={A transitivity result for ad-nilpotent ideals in type {A}},
        date={2021},
        ISSN={0019-3577},
     journal={Indagationes Mathematicae},
         url={https://www.sciencedirect.com/science/article/pii/S0019357721000434},
}

\bib{G2}{article}{
      author={Gasharov, V.},
       title={Incomparability graphs of (3+1)-free posets are {$s$}-positive},
        date={1996},
     journal={Discrete Math.},
      number={157},
       pages={211\ndash 215},
}

\bib{Gerstenhaber}{article}{
      author={Gerstenhaber, M.},
       title={Dominance over the classical groups},
        date={1961},
        ISSN={0003-486X},
     journal={Ann. of Math. (2)},
      volume={74},
       pages={532\ndash 569},
         url={https://doi.org/10.2307/1970297},
}

\bib{MGP}{article}{
      author={Guay-Paquet, M.},
       title={A modular relation for the chromatic symmetric functions of (3+1)-free posets},
        date={2013},
     journal={arXiv preprint arXiv:1306.2400},
}

\bib{GPMR}{article}{
      author={Guay-Paquet, M.},
      author={Morales, A.~H.},
      author={Rowland, E.},
       title={Structure and enumeration of {$(3+1)$}-free posets},
        date={2014},
        ISSN={0218-0006},
     journal={Ann. Comb.},
      volume={18},
      number={4},
       pages={645\ndash 674},
         url={https://doi.org/10.1007/s00026-014-0249-2},
}

\bib{gelinas2022proof}{article}{
      author={G{\'e}linas, F.},
      author={Segovia, A.},
      author={Thomas, H.},
       title={Proof of a conjecture of {M}atherne, {M}orales, and {S}elover on encodings of unit interval orders},
        date={2022},
     journal={arXiv preprint arXiv:2212.12171},
}

\bib{JH}{article}{
      author={Haglund, J.},
       title={Further investigations involving rook polynomials with only real zeros},
        date={2000},
        ISSN={0195-6698},
     journal={European J. Combin.},
      volume={21},
      number={8},
       pages={1017\ndash 1037},
         url={https://doi.org/10.1006/eujc.2000.0422},
}

\bib{HagqtCat}{book}{
      author={Haglund, J.},
       title={The {$q$},{$t$}-{C}atalan numbers and the space of diagonal harmonics},
      series={University Lecture Series},
   publisher={American Mathematical Society, Providence, RI},
        date={2008},
      volume={41},
}

\bib{hikita2024proofstanleystembridgeconjecture}{article}{
      author={Hikita, T.},
       title={A proof of the {S}tanley-{S}tembridge conjecture},
        date={2024},
     journal={arXiv preprint arXiv:2410.12758},
}

\bib{HLP}{book}{
      author={Hardy, G.~H.},
      author={Littlewood, J.~E.},
      author={P\'{o}lya, G.},
       title={Inequalities},
      series={Cambridge Mathematical Library},
   publisher={Cambridge University Press, Cambridge},
        date={1988},
        ISBN={0-521-35880-9},
        note={Reprint of the 1952 edition},
}

\bib{HMMStD}{article}{
      author={Huh, J.},
      author={Matherne, J.~P.},
      author={M\'{e}sz\'{a}ros, K.},
      author={St.~Dizier, A.},
       title={Logarithmic concavity of {S}chur and related polynomials},
        date={2022},
        ISSN={0002-9947,1088-6850},
     journal={Trans. Amer. Math. Soc.},
      volume={375},
      number={6},
       pages={4411\ndash 4427},
         url={https://doi.org/10.1090/tran/8606},
}

\bib{HOW}{incollection}{
      author={Haglund, J.},
      author={Ono, K.},
      author={Wagner, D.~G.},
       title={Theorems and conjectures involving rook polynomials with only real zeros},
        date={1999},
   booktitle={Topics in {N}umber {T}heory ({U}niversity {P}ark, {PA}, 1997)},
      series={Math. Appl.},
      volume={467},
   publisher={Kluwer Acad. Publ., Dordrecht},
       pages={207\ndash 221},
}

\bib{HaPre}{article}{
      author={Harada, M.},
      author={Precup, M.~E.},
       title={The cohomology of abelian {H}essenberg varieties and the {S}tanley--{S}tembridge conjecture},
        date={2019},
     journal={Algebr. Comb.},
      volume={2},
      number={6},
       pages={1059\ndash 1108},
}

\bib{h12}{article}{
      author={Huh, J.},
       title={Milnor numbers of projective hypersurfaces and the chromatic polynomial of graphs},
        date={2012},
        ISSN={0894-0347},
     journal={J. Amer. Math. Soc.},
      volume={25},
      number={3},
       pages={907\ndash 927},
         url={https://doi-org.libproxy.uoregon.edu/10.1090/S0894-0347-2012-00731-0},
}

\bib{HW}{article}{
      author={Haglund, J.},
      author={Wilson, A.~T.},
       title={Macdonald polynomials and chromatic quasisymmetric functions},
        date={2020},
     journal={Electron. J. Combin.},
      volume={27},
      number={3},
}

\bib{kaplansky1946problem}{article}{
      author={Kaplansky, I.},
      author={Riordan, J.},
       title={The problem of the rooks and its applications},
        date={1946},
     journal={Duke Mathematical Journal},
      volume={13},
      number={2},
       pages={259\ndash 268},
}

\bib{LZ}{article}{
      author={Lewis, J.~B.},
      author={Zhang, Y.~X},
       title={Enumeration of graded {$(\bold{3}+\bold{1})$}-avoiding posets},
        date={2013},
     journal={J. Combin. Theory Ser. A},
      volume={120},
      number={6},
       pages={1305\ndash 1327},
}

\bib{Mac}{book}{
      author={Macdonald, I.~G.},
       title={Symmetric functions and {H}all polynomials},
     edition={Second},
      series={Oxford Classic Texts in the Physical Sciences},
   publisher={The Clarendon Press, Oxford University Press, New York},
        date={2015},
}

\bib{mason}{inproceedings}{
      author={Mason, J.~H.},
       title={Matroids: unimodal conjectures and {M}otzkin's theorem},
        date={1972},
   booktitle={Combinatorics ({P}roc. {C}onf. {C}ombinatorial {M}ath., {M}ath. {I}nst., {O}xford, 1972)},
       pages={207\ndash 220},
}

\bib{McMo}{article}{
      author={McDonald, L.~M.},
      author={Moffatt, I.},
       title={On the {P}otts model partition function in an external field},
        date={2012},
     journal={J. Stat. Phys.},
      volume={146},
      number={6},
       pages={1288\ndash 1302},
}

\bib{Mphd}{thesis}{
      author={Monical, C.},
       title={Polynomials in algebraic combinatorics},
        type={Ph.D. Thesis},
        date={2018},
         url={https://www.ideals.illinois.edu/bitstream/handle/2142/101462/MONICAL-DISSERTATION-2018.pdf?sequence=1&isAllowed=y},
}

\bib{MTY}{article}{
      author={Monical, C.},
      author={Tokcan, N.},
      author={Yong, A.},
       title={Newton polytopes in algebraic combinatorics},
        date={2019},
        ISSN={1022-1824},
     journal={Selecta Math. (N.S.)},
      volume={25},
      number={5},
       pages={Paper No. 66, 37},
         url={https://doi.org/10.1007/s00029-019-0513-8},
}

\bib{M03}{book}{
      author={Murota, K.},
       title={Discrete convex analysis},
      series={SIAM Monographs on Discrete Mathematics and Applications},
   publisher={Society for Industrial and Applied Mathematics (SIAM), Philadelphia, PA},
        date={2003},
        ISBN={0-89871-540-7},
         url={https://doi.org/10.1137/1.9780898718508},
}

\bib{OS}{article}{
      author={Orellana, R.},
      author={Scott, G.},
       title={Graphs with equal chromatic symmetric functions},
        date={2014},
     journal={Discrete Math.},
      volume={320},
       pages={1\ndash 14},
}

\bib{PosGP}{article}{
      author={Postnikov, A.},
       title={Permutohedra, associahedra, and beyond},
        date={2009},
     journal={International Mathematics Research Notices},
      volume={2009},
      number={6},
       pages={1026\ndash 1106},
}

\bib{Rado}{article}{
      author={Rado, R.},
       title={An inequality},
        date={1952},
        ISSN={0024-6107},
     journal={J. London Math. Soc.},
      volume={27},
       pages={1\ndash 6},
         url={https://doi.org/10.1112/jlms/s1-27.1.1},
}

\bib{FindStat}{misc}{
      author={Rubey, M.},
      author={Stump, C.},
      author={others},
       title={{FindStat} - {T}he combinatorial statistics database},
         how={\url{http://www.FindStat.org}},
         url={http://www.FindStat.org},
        note={\url{http://www.FindStat.org}, Accessed: \today},
}

\bib{Sage-Combinat}{misc}{
      author={{S}age-{C}ombinat community, The},
       title={{S}age-{C}ombinat: enhancing {S}age as a toolbox for computer exploration in algebraic combinatorics},
         url={http://combinat.sagemath.org},
}

\bib{ShrijverA}{book}{
      author={Schrijver, A.},
       title={Combinatorial optimization. {P}olyhedra and efficiency. {V}ol. {A}},
      series={Algorithms and Combinatorics},
   publisher={Springer-Verlag, Berlin},
        date={2003},
      volume={24},
        ISBN={3-540-44389-4},
        note={Paths, flows, matchings, Chapters 1--38},
}

\bib{StSt}{article}{
      author={Stanley, R.~P.},
      author={Stembridge, J.~R.},
       title={On immanants of {J}acobi-{T}rudi matrices and permutations with restricted position},
        date={1993},
     journal={J. Combin. Theory Ser. A},
      volume={62},
      number={2},
       pages={261\ndash 279},
}

\bib{EC1}{book}{
      author={Stanley, R.~P.},
       title={Enumerative combinatorics. {V}olume 1},
     edition={Second},
      series={Cambridge Studies in Advanced Mathematics},
   publisher={Cambridge University Press, Cambridge},
        date={2012},
      volume={49},
        ISBN={978-1-107-60262-5},
}

\bib{Catbook}{book}{
      author={Stanley, R.~P.},
       title={Catalan numbers},
   publisher={Cambridge University Press, New York},
        date={2015},
        ISBN={978-1-107-42774-7; 978-1-107-07509-2},
         url={https://doi.org/10.1017/CBO9781139871495},
}

\bib{StChrom1}{article}{
      author={Stanley, R.~P.},
       title={A symmetric function generalization of the chromatic polynomial of a graph},
        date={1995},
        ISSN={0001-8708},
     journal={Adv. Math.},
      volume={111},
      number={1},
       pages={166\ndash 194},
         url={https://doi.org/10.1006/aima.1995.1020},
}

\bib{StChrom2}{incollection}{
      author={Stanley, R.~P.},
       title={Graph colorings and related symmetric functions: ideas and applications: a description of results, interesting applications, \& notable open problems},
        date={1998},
      volume={193},
       pages={267\ndash 286},
         url={https://doi.org/10.1016/S0012-365X(98)00146-0},
        note={Selected papers in honor of Adriano Garsia (Taormina, 1994)},
}

\bib{EC2}{book}{
      author={Stanley, R.~P.},
       title={Enumerative combinatorics. {V}olume 2},
      series={Cambridge Studies in Advanced Mathematics},
   publisher={Cambridge University Press, Cambridge},
        date={1999},
      volume={62},
        ISBN={0-521-56069-1; 0-521-78987-7},
         url={https://doi.org/10.1017/CBO9780511609589},
}

\bib{ShW2}{article}{
      author={Shareshian, J.},
      author={Wachs, M.~L.},
       title={Chromatic quasisymmetric functions},
        date={2016},
     journal={Adv. Math.},
      volume={295},
       pages={497\ndash 551},
}

\bib{SY}{article}{
      author={Sazdanovic, R.},
      author={Yip, M.},
       title={A categorification of the chromatic symmetric function},
        date={2018},
     journal={J. of Combin. Theory, Ser. A},
      volume={154},
       pages={218\ndash 246},
}

\bib{Valiant}{article}{
      author={Valiant, L.~G.},
       title={The complexity of computing the permanent},
        date={1979},
        ISSN={0304-3975},
     journal={Theoret. Comput. Sci.},
      volume={8},
      number={2},
       pages={189\ndash 201},
         url={https://doi.org/10.1016/0304-3975(79)90044-6},
}

\bib{Wagner}{article}{
      author={Wagner, D.~G.},
       title={Multivariate stable polynomials: theory and applications},
        date={2011},
     journal={Bull. Amer. Math. Soc. (N.S.)},
      volume={48},
      number={1},
       pages={53\ndash 84},
}

\bib{w32}{article}{
      author={Whitney, H.},
       title={The coloring of graphs},
        date={1932},
     journal={Ann. of Math. (2)},
      volume={33},
      number={4},
       pages={688\ndash 718},
}

\bib{Whi}{article}{
      author={White, D.~E.},
       title={Monotonicity and unimodality of the pattern inventory},
        date={1980},
        ISSN={0001-8708},
     journal={Adv. in Math.},
      volume={38},
      number={1},
       pages={101\ndash 108},
}

\bib{SchCbook}{book}{
      editor={Zhang, F.},
       title={The {S}chur complement and its applications},
      series={Numerical Methods and Algorithms},
   publisher={Springer-Verlag, New York},
        date={2005},
      volume={4},
        ISBN={0-387-24271-6},
}

\end{biblist}
\end{bibdiv}
\bibliographystyle{plain}
\end{document}